\documentclass[11pt, a4paper]{article}

\usepackage{microtype}

\usepackage{thmtools, thm-restate}

\usepackage[utf8]{inputenc} 
\usepackage[T1]{fontenc} 
\usepackage{geometry} 
\usepackage[%
    colorlinks      = true,
  linktocpage     = true,
  linkcolor       = blue,
  citecolor       = blue,
  urlcolor        = blue,
  breaklinks =true,
  hypertexnames=false,
]{hyperref}

\usepackage{setspace}
\usepackage[affil-it]{authblk}

\usepackage{amsmath, amsthm, amsfonts, amssymb, amsbsy, bigstrut, graphicx, enumerate,  upref, longtable, comment, booktabs, array, caption, subcaption}

\usepackage{appendix}

\usepackage{tikz}
\usepackage{xcolor}
\definecolor{lightblue}{RGB}{173,216,230}
\usetikzlibrary{shapes.geometric}
\usepackage{subcaption}

\usepackage[numbers, sort&compress]{natbib}

\newcommand{\vol}{\mathrm{vol}}

\newtheorem{thm}{Theorem}[section]
\newtheorem{lmm}[thm]{Lemma}
\newtheorem{cor}[thm]{Corollary}

\theoremstyle{definition}

\newtheorem{ex}[thm]{Example}

\newcommand{\var}{\mathrm{Var}}

\numberwithin{equation}{section}
\newcommand{\tx}{\tilde{X}}

\usepackage[utf8]{inputenc}
\usepackage{amsmath}
\usepackage{amsfonts}
\usepackage{bbm}
\usepackage{mathtools}
\usepackage{tikz}
\usetikzlibrary{decorations.pathreplacing}
\usepackage{amsthm}
\usepackage[english]{babel}
\usepackage{fancyhdr}
\usepackage{amssymb}
\usepackage{enumitem}
\usepackage{qtree}
\usepackage{mathrsfs}

\newcommand{\R}{\mathbb{R}}

\newcommand{\E}{\mathbb{E}}

\renewcommand{\P}{\mathbb{P}}

\renewcommand{\tilde}{\widetilde}
\renewcommand{\hat}{\widehat}

\newcommand{\tY}{\tilde{Y}}

\DeclareMathOperator{\conv}{conv}

\usepackage{tikz-cd}
\usetikzlibrary{decorations.pathreplacing,decorations.markings}
\usetikzlibrary{matrix}
\usetikzlibrary{arrows}
\usetikzlibrary{positioning}
\usetikzlibrary{shapes.geometric}
\usetikzlibrary{shapes.misc}
\tikzset{
  on each segment/.style={
    decorate,
    decoration={
      show path construction,
      moveto code={},
      lineto code={
        \path [#1]
        (\tikzinputsegmentfirst) -- (\tikzinputsegmentlast);
      },
      curveto code={
        \path [#1] (\tikzinputsegmentfirst)
        .. controls
        (\tikzinputsegmentsupporta) and (\tikzinputsegmentsupportb)
        ..
        (\tikzinputsegmentlast);
      },
      closepath code={
        \path [#1]
        (\tikzinputsegmentfirst) -- (\tikzinputsegmentlast);
      },
    },
  },
  mid arrow/.style={postaction={decorate,decoration={
        markings,
        mark=at position .5 with {\arrow[#1]{stealth}}
      }}},
}

\tikzset{->-/.style={decoration={
  markings,
  mark=at position .5 with {\arrow{>}}},postaction={decorate}}}

\newcommand{\tX}{\tilde{X}}

\begin{document}
\title{Estimating the size of a set using cascading exclusion}

\author{Sourav Chatterjee\thanks{Department of Statistics and Department of Mathematics, Stanford University, USA. Email: \href{mailto:souravc@stanford.edu}{\tt souravc@stanford.edu}. 
}}
\author{Persi Diaconis\thanks{Department of Statistics and Department of Mathematics, Stanford University, USA. Email: \href{mailto:diaconis@math.stanford.edu}{\tt diaconis@math.stanford.edu}. 
}}
\author{Susan Holmes\thanks{Department of Statistics, Stanford University, USA. Email: \href{mailto:sph@stanford.edu}{\tt sph@stanford.edu}. 
}}
\affil{Stanford University}


\maketitle


\begin{abstract}
Let $S$ be a finite set, and $X_1,\ldots,X_n$ an i.i.d. uniform sample from $S$. To estimate the size $|S|$, without further structure, one can wait for repeats and use the birthday problem. This requires a sample size of the order $|S|^\frac{1}{2}$. On the other hand, if $S=\{1,2,\ldots,|S|\}$, the maximum of the sample blown up by $n/(n-1)$ gives an efficient estimator based on any growing sample size. This paper gives refinements that interpolate between these extremes. A general non-asymptotic theory is developed. This includes estimating the volume of a compact convex set, the unseen species problem, and a host of testing problems that follow from the question `Is this new observation a typical pick from a large prespecified population?'
We also treat regression style predictors. A general theorem gives non-parametric finite $n$ error bounds in all cases.
\end{abstract}

\tableofcontents

\section{Introduction}
\subsection{Background}
Let $S$ be a finite set and $X_1,X_2,\ldots,X_n$ an independent sample from the uniform distribution on $S$. One wants to estimate $|S|$, the size of $S$. With no further structure, a natural procedure is to wait for repeats and use the birthday problem. Let $T$ be the time of the first repeat. The birthday computation gives
$$\P(T=n) = \frac{n-1}{|S|}\prod_{j=1}^{n-2}\biggl(1-\frac{j}{|S|}\biggr) \sim \frac{n-1}{|S|}\exp\biggl(-\frac{(n-1)(n-2)}{2|S|}\biggr).$$
Treating $|S|$ as a real parameter, take logs and differentiate the logarithm of the above quantity with respect to $|S|$ to get the natural estimator,
\begin{equation}
\label{eq1.1}
    \hat{|S|} = \frac{T(T-1)}{2}.
\end{equation}                     
Repeats require a sample of size $|S|^{1/2}$ which can be impossible if $|S|$ is very large. 

At the other extreme, suppose $S=\{1,2,\ldots,|S|\}$. Then, the sample maximum scaled up by $n/(n-1)$ is a good estimate:
\begin{equation}
\label{eq1.2}
\hat{|S|} = \frac{n}{n-1}\max_{1\le i\le n} X_i.
\end{equation}
This works well for any growing sample size. This is the `German tank problem' extensively used in war-time (see the Wikipedia entry on the German tank problem~\cite{wikipedia:german_tank_problem}). The present paper offers a wealth of ways to interpolate between these extremes using partial orders on $S$. The development is more general, allowing estimates of volume in continuous settings.

There are many examples where uniform samples from $S$ are available but the size $|S|$ is unknown, and of interest. For example, if $S$ is the set of $I \times J$ contingency tables with given row and column sums, Monte Carlo Markov Chain methods give easy access to uniform samples~\cite{diaconis1998algebraic} but $|S|$ is unknown. Further recent works use the Burnside process for enumerating the number of orbits of a finite group acting on a finite set~\cite{jerrum1993uniform,diaconis2025counting,diaconis2020hahn}. This includes contingency tables as a special case \cite{diaconis2025random}. More examples where uniform samples are easily accessible but $|S|$ is unknown are provided in sections \ref{sec:convexhull}, \ref{sec:uppersets},
\ref{sec:convsubsets}  below.

\subsection{An abstract setting}
Let $(S, \mathcal{S})$ be a measurable space, and let $2^S$ be the power set of $S$. We will say that a set-valued map $A:S^n \to 2^S$ is measurable if the map $(x_1,\ldots,x_{n+1}) \mapsto 1_{\{x_{n+1}\in A(x_1,\ldots,x_n)\}}$ is measurable. We will say that $A$ is symmetric if it is invariant under permutations of coordinates. If $A$ and $B$ are two measurable set-valued maps, then so is $A\cap B$, because the indicator of the intersection is the product of the indicators. By the inclusion-exclusion formula, this implies that $A\cup B$ is measurable, and therefore, $A\setminus B = (A\cup B) \setminus B$ is measurable. Consequently, the symmetric difference $A\Delta B$ is measurable. A set-valued map $A$ from $S^n$ can also be viewed as a set-valued map from $S^m$ for any $m>n$, by defining $A:S^m \to 2^S$ as $A(x_1,\ldots,x_m) := A(x_1,\ldots,x_n)$. Lastly, observe that if $A:S^n\to 2^S$ is a measurable set-valued map, then the map $(x_1,\ldots,x_n)\mapsto \mu(A(x_1,\ldots,x_n))$ is measurable, since
\[
\mu(A(x_1,\ldots,x_n)) = \int_S 1_{\{x_{n+1}\in A(x_1,\ldots,x_n)\}} d\mu(x_{n+1}).
\]
We will use these observations freely below. 
\begin{thm}\label{mainthm}
Let $(S, \mathcal{S})$ be a measurable space and let $X_1,\ldots,X_n$ be i.i.d.~$S$-valued random variables with law $\mu$, where $n\ge 3$. Let $A:S^n \to 2^S$,  $A':S^{n-1} \to 2^S$ and $A'' :S^{n-2}\to 2^S$ be three symmetric set-valued maps that are measurable in the above sense. Define 
\begin{align*}
&\theta := \E[\mu(A'(X_1,\ldots,X_{n-1}))(1-\mu(A'(X_1,\ldots,X_{n-1})))],\\
&\delta' := \E[\mu(A(X_1,\ldots,X_n) \Delta A'(X_1,\ldots,X_{n-1}))],\\
&\delta'' := \E[\mu(A'(X_1,\ldots,X_{n-1})\Delta A''(X_1,\ldots,X_{n-2}))].
\end{align*}
Then 
\begin{align*}
&\E\biggl[\biggl(\mu(A(X_1,\ldots,X_n)) - \frac{1}{n}\sum_{i=1}^n 1_{\{X_i \in A'(X_1,\ldots,X_{i-1},X_{i+1},\ldots,X_n)\}}\biggr)^2\biggr]\\
&\le 2\delta' + \frac{4(n-1)\delta''}{n} + \frac{2\theta}{n}.
\end{align*}
\end{thm}
A proof of Theorem \ref{mainthm} appears in Appendix \ref{mainproof}.  The theorem says, roughly speaking, that if the random set $A(X_1,\ldots,X_n)$ is well-approximated by a random set $A'(X_1,\ldots,X_{n-1})$, and $A'(X_1,\ldots,X_{n-1})$ is well-approximated by $A''(X_1,\ldots,X_{n-2})$, and if $n$ is large, then the measure of $A(X_1,\ldots,X_n)$ can be accurately estimated by the leave-one-out estimator
\[
\frac{1}{n}\sum_{i=1}^n 1_{\{X_i \in A'(X_1,\ldots,X_{i-1},X_{i+1},\ldots,X_n)\}}.
\]
The key assumptions that make this happen are that the functions $A$, $A'$ and $A''$ are all symmetric in their arguments, and that $X_1,\ldots,X_n$ are i.i.d.

We see here that the exclusion or leave-one-out sums are not built to estimate future performance but to estimate the current measure
$\mu(A(X_1,\ldots,X_n))$ by leveraging the stability between $A$, $A'$, and $A''$ under systematic exclusion.
We like to call this idea {\bf cascading exclusion}.

Section \ref{appsec} develops applications and gives sharp, finite sample bounds for $\theta$, $\delta'$ and $\delta''$ tailored to these applications.

Section \ref{sec:unseen} treats the unseen species problem. There, $S$ is a finite or countable set and $\P(X_i=s)=p(s)$ is an unknown probability $p$ on $S$. Taking $A = A(X_1,\ldots,X_n) := S \setminus \{X_1,\ldots,X_n\}$, we get
$$\mu(A) = n_0 := \sum_{s \notin  \{X_1,...,X_n\}} p(s),$$
the chance of seeing a new observation in the next sample. Our estimator becomes the Good--Turing estimator
$$\hat{n}_0 = \frac{n_1}{n}$$
with $n_1$ the number of singletons, i.e., the observations occurring once in the sample.

The limits in Theorem \ref{mainthm} are developed to prove Lemma \ref{unseenlemma}, which implies by Theorem~\ref{unseencor2}
that $\E[(n_0 - \hat{n}_0)^2] \leq \frac{5}{n-2}$.
Section \ref{sec:convexhull} gives the promised interpolation between the `unknown $k$' birthday problem and the German tank problem in the presence of a partial order on $S$.
Section \ref{sec:uppersets} treats the continuous problem of estimating the volume of a convex set. It gives explicit finite sample accuracy bounds
that are useful for any dimension $d$ provided that roughly $n \gg d$.
Section \ref{aldoussec} treats a problem suggested by David Aldous. There, $|S|$ is a well defined corpus --- e.g., a library of DNA sequences or a collection of melodies. One wants to test if a new entry is suspiciously close to any element of the target collection. The setup uses a distance on $S$ and gives a universally valid test.
Section~\ref{predictsec} treats a general prediction problem
where one observes $\{ (X_i,Y_i)\}_{i=1,\ldots,n}$, builds a prediction interval for $Y$ given $X$ using some given procedure, and then estimates the coverage probability of the prediction interval using a leave-one-out estimator. Section \ref{devroyesec} draws connection to an old theorem of Devroye and Wagner on estimating misclassification rates.

The `leave-one-out' estimators underlying Theorem \ref{mainthm} are classical statistics fare, often studied as `cross-validation' (although with a fundamentally different objective than ours, as explained above); see \cite{stone1978cross} for history and references. 
\citet{shao1993linear} proves limit theorems for cross validation in the regression setting, but we emphasize that the present theorems
are non-asymptotic with explicit constants.
As indicated, the literature on applications is huge and we point our readers to reviews in the different sections.

All sections contain worked examples with available code (see github repository here: {\tt https://spholmes.github.io/SetSize/}). All contain reviews of the literature and can be read now for further motivation.

Coming back to Theorem \ref{mainthm}, one remark is that the appearance of the $n-2$ object $A''$ may seem stronger than necessary. In fact, a weaker variant of Theorem \ref{mainthm} can be proved using only $A'$: 
\begin{thm}\label{mainthm2}
In the setting of Theorem \ref{mainthm}, let
\[
\rho:=E\big[\mu(A'(X_1,\ldots,X_{n-1})\Delta A'(X_1,\ldots,X_{n-2},X_n))\big].
\]
Let $\delta'$ and $\theta$ be as in  Theorem \ref{mainthm} and $\rho$ be as above. Then we have 
\[
\E\biggl[\biggl(\mu(A(X_1,\ldots,X_n))
-\frac1n\sum_{i=1}^n
1_{\{X_i\in A'(X_1,\ldots,X_{i-1},X_{i+1},\ldots,X_n)\}}
\biggr)^2\biggr]
\le
2\delta'+\frac{4(n-1)}{n}\rho+\frac{2\theta}{n}.
\]
Moreover, by symmetry and the triangle inequality, $\rho\le 2\delta'$, and therefore
\[
\E\biggl[\biggl(\mu(A(X_1,\ldots,X_n))
-\frac1n\sum_{i=1}^n
1_{\{X_i\in A'(X_1,\ldots,X_{i-1},X_{i+1},\ldots,X_n)\}}
\biggr)^2\biggr]
\le
\left(10-\frac{8}{n}\right)\delta'+\frac{2\theta}{n}.
\]
\end{thm}
This theorem is proved in Appendix \ref{mainproof2}. We nevertheless state Theorem \ref{mainthm} in the stronger three-level form because the $n-2$ object cleanly isolates a second-order dependence term between distinct leave-one-out summands, and in the applications this typically gives a sharper and more natural bound.

It may be helpful to view Theorem \ref{mainthm} as saying that \textit{stability plays the role of structure}. In many problems one would ordinarily control $\mu(A(X_1,\ldots,X_n))$ by imposing a parametric model, a prior distribution, or some other explicit regularity on $\mu$. Here the inferential content comes instead from symmetry and near-invariance: if deleting one or two observations does not substantially change the random set, then the leave-one-out average becomes a reliable proxy for its present $\mu$-measure. In this sense, Theorem \ref{mainthm} isolates a general mechanism by which stability can replace more detailed modelling assumptions. This viewpoint is congenial to Bayesian thinking, except that the prior-like content here is carried not by a distribution on $\mu$ but by a structural regularity statement about what happens when a few observations are deleted.

This objective is different from several familiar deletion-based procedures. Cross-validation \cite{stone1978cross} uses deletion to estimate future predictive performance of a method trained on a reduced sample; jackknife-type arguments use deletion to assess bias or variance; and in empirical-Bayes settings one sometimes deletes an observation to avoid using it twice. Here deletion is used for none of these purposes. The target is the current random quantity $\mu(A(X_1,\ldots,X_n))$ itself, and exclusion is a device for revealing it through the stability of the underlying set-valued maps.

The scope of the theorem is also worth stressing. The result does not require moment assumptions---heavy tails are not an obstacle by themselves---but it does require i.i.d.\ sampling, permutation symmetry, and small stability terms $\delta'$ and $\delta''$. When dependence is present, when symmetry is absent, or when the functional is highly sensitive to rare configurations, these terms need not be small and the leave-one-out estimator may fail. A final caution is that Theorem \ref{mainthm} is not a concentration result for $\mu(A(X_1,\ldots,X_n))$: the theorem says that the leave-one-out estimator tracks this random quantity well, not that the random quantity is close to a deterministic limit. Appendix \ref{subtle} gives an example due to Aldous where the error bound is small but $\mu(A(X_1,\ldots,X_n))$ still has a genuinely random limit.

\section{Applications}\label{appsec}
\subsection{Unseen species}
\label{sec:unseen}
An island has $N$ unknown species of animals, where $N$ is allowed to be infinity. A zoologist, at every turn, observes an animal from species $i$ with probability $p_i$, where $p_1,\ldots,p_N$ are nonnegative and add up to $1$. The $p_i$'s and the number $N$ are unknown. Given the data that the zoologist has at the end of $n$ turns, what is an estimate of the probability that the zoologist observes an animal from a new species at the next turn? The following theorem, which we obtain as a corollary of Theorem \ref{mainthm}, shows that the classical Good--Turing estimate $\hat{n}_0$ (which we denote by $T_n/n$ below) is accurate whenever $n$ is large. 
\begin{thm}\label{unseencor2}
In the above setting, let $T_n$ be the number of species that have been observed exactly once up to time $n$, where $n\ge 3$, and let $P_n$ be the conditional probability, given the history up to time $n$, that a new species is observed at time $n+1$.
We have the following bound that depends solely on~$n$:
\begin{align*}
\E\biggl[\biggl(\frac{T_n}{n} - P_n\biggr)^2\biggr]\le \frac{2}{e(n-1)} + \frac{4(n-1)}{en(n-2)} + \frac{2}{n}\le \frac{5}{n-2}.
\end{align*}
\end{thm}
Theorem \ref{unseencor2} will be proven in several steps. Some of these offer different bounds. First, Lemmas \ref{unseenlemma} and \ref{unseenlemma2}, then finally the proof of Theorem \ref{unseencor2}.

\begin{lmm}\label{unseenlemma}
In the above setting, we have 
\begin{align*}
\E\biggl[\biggl(\frac{T_n}{n} - P_n\biggr)^2\biggr] &\le 2\sum_{i=1}^N p_i^2(1-p_i)^{n-1} + \frac{4(n-1)}{n}\sum_{i=1}^N p_i^2(1-p_i)^{n-2} \\
&\qquad + \frac{2}{n}\sum_{i=1}^N p_i(1-p_i)^{n-1}.
\end{align*}
\end{lmm}
\begin{proof}
Let $X_i\in\{1,\ldots,N\}$ be the species seen at time $i$. Let $\mu$ be the law of the $X_i$'s (i.e., $\mu(\{j\}) = p_j$ for each $j$). Let $A(X_1,\ldots,X_n)$ be the set of species that have not been observed up to time $n$, and define $A'$ and $A''$ similarly, replacing $n$ by $n-1$ and $n-2$, respectively. Then 
\[
\mu(A(X_1,\ldots,X_n)) = P_n,
\]
and
\begin{align*}
\frac{1}{n}\sum_{i=1}^n 1_{\{X_i \in A'(X_1,\ldots,X_{i-1},X_{i+1},\ldots,X_n)\}} = \frac{T_n}{n}.
\end{align*}
Thus, we only have to compute upper bounds on the quantities $\theta$, $\delta'$ and $\delta''$ from Theorem~\ref{mainthm}. First, note that
\begin{align*}
\theta &\le \E[\mu(A'(X_1,\ldots,X_{n-1}))]\\
&= \P(X_n\notin \{X_1,\ldots,X_{n-1}\})\\
&= \sum_{i=1}^N \P(X_j\ne i \text{ for all } j\le n-1 |X_n = i)\P(X_n=i)\\
&= \sum_{i=1}^N p_i(1-p_i)^{n-1}.
\end{align*}
Similarly, 
\begin{align*}
\delta' &= \P(X_{n+1}\in A(X_1,\ldots,X_n)\Delta A'(X_1,\ldots,X_{n-1}))\\
&= \P(X_{n+1} = X_n \text{ and } X_{n+1} \ne X_j \text{ for all } j\le n-1)= \sum_{i=1}^N p_i^2(1-p_i)^{n-1}.
\end{align*}
The same calculation gives
\[
\delta'' = \sum_{i=1}^N p_i^2(1-p_i)^{n-2}.
\]
By Theorem \ref{mainthm}, this completes the proof.
\end{proof}


\begin{proof}[Proof of Theorem \ref{unseencor2}]
Take any positive integer $m$. It is easy to show  that the maximum value of $x(1-x)^m$ as $x$ ranges over $[0,1]$ is attained when $x=1/(m+1)$, where the value is
\[
\frac{1}{m+1}\biggl(1-\frac{1}{m+1}\biggr)^m = \frac{1}{m}\biggl(1-\frac{1}{m+1}\biggr)^{m+1}\le \frac{1}{em}.
\]
This shows that the upper bound from Lemma \ref{unseenlemma} is bounded above by
\begin{align*}
2\sum_{i=1}^N \frac{p_i}{e(n-1)} + \frac{4(n-1)}{n}\sum_{i=1}^N \frac{p_i}{e(n-2)} + \frac{2}{n}\sum_{i=1}^N p_i\le \frac{2}{e(n-1)} + \frac{4(n-1)}{en(n-2)} + \frac{2}{n}. 
\end{align*}
This completes the proof.
\end{proof}
In spite of the simplicity of the bound in Theorem~\ref{unseencor2}, the bound can be improved if $N\ll n$, as shown by the following lemma (which is a corollary of Lemma \ref{unseenlemma}).
\begin{lmm}\label{unseenlemma2}
In the setting of Lemma \ref{unseenlemma}, we have 
\begin{align*}
\E\biggl[\biggl(\frac{T_n}{n} - P_n\biggr)^2\biggr] &\le \frac{4N}{(n-2)^2}
\end{align*}
\end{lmm}

\begin{proof}
Using the bound $1-x \le e^{-x}$, we get that the bound in Lemma \ref{unseenlemma} is bounded above by $\sum_{i=1}^N g(p_i)$, where 
\begin{align*}
g(p) &:= 2 p^2(1-p) e^{-(n-2)p} + \frac{4(n-1)}{n}p^2e^{-(n-2)p} + \frac{2}{n}p(1-p)e^{-(n-2)p}\\
&\le \biggl(6p^2 + \frac{2p}{n}\biggr) e^{-(n-2)p}.
\end{align*}
Writing $t := (n-2)p$, the last expression can be expressed as
\[
\biggl(\frac{6t^2}{(n-2)^2} + \frac{2t}{n(n-2)}\biggr) e^{-t} \le \frac{(6t^2 + 2t)e^{-t}}{(n-2)^2}.
\]
Now, $t^2 e^{-t}$ is maximized at $t=2$ and $t e^{-t}$ is maximized at $t=1$. Thus, for any $t>0$,
\[
(6t^2 + 2t)e^{-t} \le 24 e^{-2} + 2e^{-1}\le 4.
\]
This gives the required bound.
\end{proof}

Theorem \ref{unseencor2} is derived from our general theory, Theorem \ref{mainthm}. For the special case of unseen species, \citet{robbins1968estimating}
has the sharper bound of $\frac{1}{n}$ to our $\frac{5}{n-2}$,
and shows that when $n$ is large,
$\frac{0.6}{n}$  is a lower bound.

\medskip
\noindent\textit{Higher occupancy classes.}
For an integer $r\ge 0$, let $T_{r,n}$ denote the number of species that have been observed exactly $r$ times up to time $n$, so that $T_{1,n}=T_n$. Define
\[
P_{n,r}:=\sum_{s=1}^N p_s\,1_{\{\sum_{j=1}^n 1_{\{X_j=s\}}=r\}},
\]
which is the conditional probability, given $X_1,\ldots,X_n$, that $X_{n+1}$ belongs to a species that has been seen exactly $r$ times among $X_1,\ldots,X_n$. Thus $P_{n,0}=P_n$. If we set
\[
A_r(X_1,\ldots,X_n):=\Bigl\{s:\sum_{j=1}^n 1_{\{X_j=s\}}=r\Bigr\},
\]
then the corresponding leave-one-out estimator is
\[
\hat{P}_{n,r}:=\frac{r+1}{n}T_{r+1,n},
\]
because
\[
\frac{1}{n}\sum_{i=1}^n 1_{\{X_i\in A_r(X_1,\ldots,X_{i-1},X_{i+1},\ldots,X_n)\}}
=
\frac{r+1}{n}T_{r+1,n}.
\]
Applying Theorem \ref{mainthm} with $A=A_r$ yields, with the convention $\binom{m}{-1}=0$,
\begin{align*}
&\E\bigl[(\hat{P}_{n,r}-P_{n,r})^2\bigr]\\ 
&\le 2\sum_{s=1}^N \biggl[\binom{n-1}{r-1}p_s^{r+1}(1-p_s)^{n-r}
+\binom{n-1}{r}p_s^{r+2}(1-p_s)^{n-1-r}\biggr]\\
&\quad+\frac{4(n-1)}{n}\sum_{s=1}^N \biggl[\binom{n-2}{r-1}p_s^{r+1}(1-p_s)^{n-1-r}
+\binom{n-2}{r}p_s^{r+2}(1-p_s)^{n-2-r}\biggr]\\
&\quad+\frac{2}{n}\sum_{s=1}^N \binom{n-1}{r}p_s^{r+1}(1-p_s)^{n-1-r}.
\end{align*}
In particular, for each fixed $r$, the same maximization argument used in the proof of Theorem~\ref{unseencor2} gives
\[
\E\bigl[(\hat{P}_{n,r}-P_{n,r})^2\bigr]\le \frac{C_r}{n}
\]
for a constant $C_r$ depending only on $r$.

\medskip

\noindent\textit{Stability and extrapolation.}
The Good--Turing functional treated above is a one-step predictor and is naturally stable under deletion of a single observation: removing one data point only affects the occupancy counts locally. This is exactly the mechanism captured by Theorem~\ref{mainthm}. The situation is different for longer extrapolation functionals, such as the classical Good--Toulmin estimator \cite{good1956number}, which uses alternating weights of order $(m/n)^r$ to estimate the number of new species in an additional sample of size $m$. When $m$ is comparable to, or larger than, $n$, small perturbations of the occupancy counts can then be strongly amplified, so the relevant functional is no longer stable under deletion in our sense. This gives one way to interpret the well-known instability of the raw Good--Toulmin estimator and suggests that substantial extrapolation requires additional regularization or prior structure. Bayesian nonparametric methods are one natural way to supply such additional regularity; see, for example, \citet{favaro2009bayesian}. One may therefore expect the specific prior to matter much more in the unstable regime than in settings where deletion stability already holds at the functional level.

\subsubsection{Literature Review}
The unseen species problem has a long history, from Jaccard's work on counting flora present in the Alps~\cite{Jaccard1902},  Corbet and Fisher's butterfly's \cite{Corbet1943}, Good and Turing's applications in World War II code breaking \cite{good2000turing} and `How many words did Shakespeare know' by \citet{ThistedEfron1976}. A 1993 survey of \citet{Bunge1993} has 180 references to earlier writing in theoretical and applied statistics. 

A whole section of literature in Bayesian nonparametrics (see \citet{Favaro2016} and \citet{lijoi2007bayesian}) has refined the use of prior information in the estimations which in practical situations is the most reasonable approach since such information 
is easily available.
There are many justifications for the estimator $n_0$ of Good--Turing (see \cite{DiaconisStein}). 

The paper of \citet{Lo1992} uses `leave one out' methodology in a similar fashion to the present development, and goes on to relate it to other problems. See the literature review in Section~\ref{sec:convexhull} for more details.

There is an extremely large literature on the comparison and estimation of the {\it number} of different
species in different environments; see \citet{Gotelli_Colwell_2010} for a book chapter reviewing the literature. Suppose the true number of species in a population is denoted $R$. \citet{darroch1980note} suggested to blow up the observed $R_{\text{obs}}$ by the coverage factor and use the
estimate $R_{\text{obs}}/(1-\frac{n_1}{N})$.
\citet{chao1984nonparametric} showed an improvement using both the doubletons and singletons.
An interesting development occurred in the study of diversity inference in ecology with the work of \citet{sanders1968marine} and followup corrections \citet{simberloff1972properties} that proposed going further than leave-one-out. Before the invention of the bootstrap by Efron, Sanders suggested the construction of rarefaction curves built by taking systematic subsamples of size $n-1,n-2,n-3,\ldots,1$ and plotting the curve of numbers of species. Simulations are not necessary here as \citet{hurlbert1971nonconcept} gives
formulae for the curves.
In statistical terms, \citet{siegel1982rarefaction}
point out that we can give confidence bands 
for the true curves determined by the observed multinomial counts thus using the data on the counts.
The extrapolation of these curves give apparently good estimates of the number of species, which have not been justified theoretically. 

\subsection{Convex hulls}
\label{sec:convexhull}

\begin{figure}[htbp]
    \centering
    \includegraphics[width=0.7\textwidth]{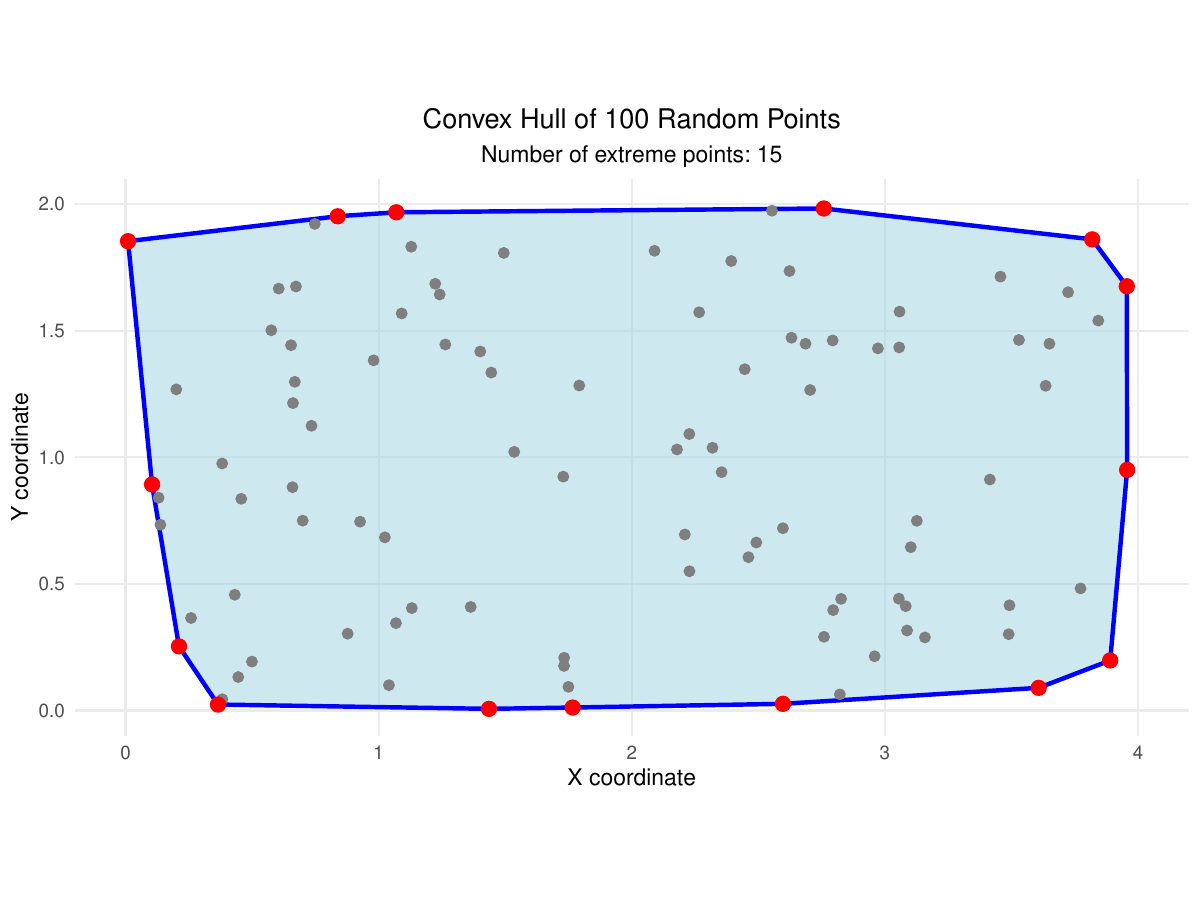}
    \caption{Convex hull of 100 randomly generated points uniformly distributed in the rectangle $[0,4] \times [0,2]$. The light blue shaded region represents the convex hull, while red points indicate the extreme points (vertices of the convex hull).}
    \label{fig:convex_hull_ex}
\end{figure}

Let $X_1,\ldots,X_n$ be i.i.d.~random points drawn from some probability measure $\mu$ on $\R^d$. Let $V_n$ be the number of vertices (i.e., extreme points) of the convex hull of $X_1,\ldots,X_n$ as illustrated in the two dimensional case in Figure \ref{fig:convex_hull_ex}. It is a simple fact, originally observed by \citet{efron65}, that 
\[
\frac{1}{n}\E(V_n)=\E(D_{n-1}), 
\]
where $D_{n-1}$ is  the $\mu$-measure of the complement of the convex hull of $X_1,\ldots,X_{n-1}$. The following theorem, which we obtain as a corollary of Theorem \ref{mainthm}, shows that if $n\gg d$, then not only are the expected values  of $\frac{1}{n} V_n$ and $D_{n-1}$ equal to each other, but the two random variables themselves are close to each other with high probability. This is applied to give a useful estimate of volume in corollary \ref{convcor}.
\begin{thm}\label{convthm}
Let $X_1,\ldots,X_n$ be i.i.d.~random points drawn from some probability measure $\mu$ on $\R^d$, where $n\ge 3$. Let $V_n$ be the number of extreme points of $\conv(X_1,\ldots,X_n)$. Let $D_n := 1 - \mu(\conv(X_1,\ldots,X_{n}))$. Then 
\[
\E\biggl[\biggl(\frac{V_n}{n} - D_n\biggr)^2\biggr] \le \frac{6d+7}{n}.
\]
and $\E|D_n - D_{n-1}| \le (d+1)/n$.
\end{thm}
\begin{proof}
To put this problem in the framework of Theorem \ref{mainthm}, let 
\[
A(x_1,\ldots,x_n) :=\R^d \setminus \conv(x_1,\ldots,x_n),
\]
and define $A'$ and $A''$ to be the same, but with $n-1$ and $n-2$ points, respectively. Then note that 
\begin{align*}
\mu(A(X_1,\ldots,X_n)) &= 1- \mu(\conv(X_1,\ldots,X_n)) = D_n, 
\end{align*}
and 
\begin{align*}
\frac{1}{n}\sum_{i=1}^n 1_{\{X_i \in A'(X_1,\ldots,X_{i-1},X_{i+1},\ldots,X_n)\}} &= \frac{1}{n}\sum_{i=1}^n 1_{\{X_i \notin \conv(X_1,\ldots,X_{i-1},X_{i+1},\ldots,X_n)\}}\\
&= \frac{V_n}{n}.
\end{align*}
Thus, to prove the first inequality claimed in the theorem, we only have to compute upper bounds on the quantities $\theta$, $\delta'$ and $\delta''$ from Theorem \ref{mainthm}. We bound $\theta$ simply by $\frac{1}{4}$. Next, note that
\begin{align*}
\delta'' &= \E[\mu(A'(X_1,\ldots,X_{n-1})\Delta A''(X_1,\ldots,X_{n-2}))]\\
&= \P[X_n \in \conv(X_1,\ldots,X_{n-1})\setminus \conv(X_1,\ldots,X_{n-2})].
\end{align*}
Let us call an index $i\in \{1,\ldots,n-1\}$ `indispensable' if $X_n$ is in  $\conv(X_1,\ldots,X_{n-1})$ but $X_n$ is not in $\conv(X_1,\ldots, X_{i-1}, X_{i+1},\ldots,X_{n-1})$. We claim that the set of indispensable vertices has size $\le d+1$ with probability one. To see this, let $P$ be the set of indispensable indices in a particular realization. Without loss, let us assume that $|P|\ge 1$, which implies that $X_n\in \conv(X_1,\ldots,X_{n-1})$. By Carath\'eodory's theorem for convex hulls in Euclidean space, there is a set $Q\subseteq \{1,\ldots,n-1\}$ of size $\le d+1$ such that $X_n$ is in the convex hull of $(X_i)_{i\in Q}$. If an index $i$ is not in $Q$, then clearly $i$ cannot be indispensable, because $Q\subseteq \{1,\ldots,i-1,i+1,\ldots,n-1\}$ and so $X_n$ is in the convex hull of $X_1,\ldots, X_{i-1}, X_{i+1},\ldots,X_{n-1}$. Thus, $P\subseteq Q$. This proves that $|P|\le d+1$. Consequently,
\[
d+1\ge \E|P| = \sum_{i=1}^{n-1} \P(i\in P). 
\]
But by symmetry, $\P(i\in P)$ is the same for each $i$. Therefore $\P(i\in P) \le (d+1)/(n-1)$ for each $i$. But $\delta'' = \P(n-1\in P)$. Thus,
\[
\delta'' \le \frac{d+1}{n-1}.
\]
By exactly the same argument with $n$ replaced by $n+1$ (and introducing an extra variable $X_{n+1}$), we get
\[
\delta'\le \frac{d+1}{n}.
\]
Thus, by Theorem \ref{mainthm}, we get the first inequality claimed in the theorem. For the second inequality, simply note that $\E|D_n - D_{n-1}| = \delta'$, and apply the bound obtained above.
\end{proof}

Theorem \ref{convthm} can be used to quantify, as follows, the error of a natural estimate of the volume of a convex set if we know how to draw uniformly from the convex set. \citet{diaconis1985testing} encountered this problem in their work on contingency tables. They used the volume of the set of positive real arrays with given row and column sums as a surrogate for the number of tables. Uniform samples are available using many varieties of the `hit and run' 
\cite{diaconis2007hit}  and sequential importance sampling algorithms
\cite{chen2005sequential}. See further discussions in \citet{diaconis2013sampling}.

Let $K$ be a bounded convex subset of $\R^d$ with nonzero volume. Let $X_1,\ldots,X_n$ be drawn independently and uniformly at random from $K$. Let $V_n$ be the number of vertices in the convex hull of $X_1,\ldots,X_n$. By Theorem \ref{convthm}, we know that if $n\gg d$, then the random variable $V_n/n$ is close to the random variable 
\[
D_n = 1 - \frac{\vol(\conv(X_1,\ldots,X_n))}{\vol(K)}
\]
with high probability. This suggests that the following would be a good estimate of $\vol(K)$:
\begin{align}\label{volest}
\hat{\vol(K)} := \frac{\vol(\conv(X_1,\ldots,X_n))}{1-\frac{V_n}{n}}.
\end{align}
To get an upper bound on the error of this estimate, observe that 
\begin{align*}
\biggl(\frac{V_n}{n} - D_n\biggr)^2 &= \biggl( \frac{\vol(\conv(X_1,\ldots,X_n))}{\vol(K)} - 1+\frac{V_n}{n}\biggr)^2\\
&= \biggl( \frac{\hat{\vol(K)}}{\vol(K)}\biggl(1-\frac{V_n}{n}\biggr) - 1+\frac{V_n}{n}\biggr)^2\\
&= \biggl(1-\frac{V_n}{n}\biggr)^2 \biggl(\frac{\hat{\vol(K)}}{\vol(K)}-1\biggr)^2.
\end{align*}
Thus, Theorem \ref{convthm} yields the following corollary.
\begin{cor}\label{convcor}
Let $X_1,\ldots,X_n$ be i.i.d.~random points drawn from the uniform distribution on $K$, where $K$ is a bounded convex subset of $\R^d$ with nonzero volume, and $n\ge 3$. Let $V_n$ be the number of extreme points of $\conv(X_1,\ldots,X_n)$. Then
\begin{align*}
\E\biggl[\frac{(n-V_n)^2}{(6d+7)n}\biggl(\frac{\hat{\vol(K)}}{\vol(K)}-1\biggr)^2\biggr]\le 1.
\end{align*}
\end{cor}
The above corollary implies that if  $n-V_n \gg \sqrt{nd}$ in a particular realization, then we can expect that $\hat{\vol(K)}$ is a good estimate of $\vol(K)$. In particular, it gives the following level $1-\alpha$ confidence interval for $\vol(K)$: 
\begin{align*}
I_\alpha := \left[\frac{\hat{\vol(K)}}{1 + \frac{\sqrt{(6d+7)n}}{\sqrt{\alpha}(n-V_n)}},\, \frac{\hat{\vol(K)}}{\max\biggl\{1 - \frac{\sqrt{(6d+7)n}}{\sqrt{\alpha}(n-V_n)},0\biggr\}} \right].
\end{align*}
To see that this indeed has level $1-\alpha$, simply note that by Markov's inequality,
\begin{align*}
\P(\vol(K)\notin I_\alpha) &= \P\biggl(\frac{\hat{\vol(K)}}{\vol(K)} \notin  \biggl[1-\frac{\sqrt{(6d+7)n}}{\sqrt{\alpha}(n-V_n)}, \, 1 + \frac{\sqrt{(6d+7)n}}{\sqrt{\alpha}(n-V_n)}\biggr]\biggr)\\
&= \P\biggl(\biggl|\frac{\hat{\vol(K)}}{\vol(K)} - 1\biggr| > \frac{\sqrt{(6d+7)n}}{\sqrt{\alpha}(n-V_n)}\biggr)\\
&= \P\biggl(\frac{(n-V_n)^2}{(6d+7)n}\biggl(\frac{\hat{\vol(K)}}{\vol(K)}-1\biggr)^2 > \frac{1}{\alpha}\biggr)\\
&\le \alpha\E\biggl[\frac{(n-V_n)^2}{(6d+7)n}\biggl(\frac{\hat{\vol(K)}}{\vol(K)}-1\biggr)^2\biggr]\le \alpha.
\end{align*}

\subsubsection{Small illustrations in low dimension}

For the example in Figure \ref{fig:convex_hull_ex}, there are 15 extreme points and the convex hull has an area of $7.266$, which gives $\hat{\vol(K)} = 8.55$, where
$\vol(K) = 8$.

Theorem \ref{convthm} is empirically validated in Figures~\ref{fig:main_theorem_mse}--\ref{fig:correlarywithdifferentsupports}. Figure~\ref{fig:main_theorem_mse} shows that the MSE bound in Theorem \ref{convthm} is valid for all tested distributions and dimensions. The stability of the $D_n$ sequence is confirmed in 
Figure~\ref{fig:main_theorem_D_diff}
showing the predicted $O(1/n)$ convergence. 

\begin{figure}[htbp]
    \centering
    \includegraphics[width=0.9\textwidth]{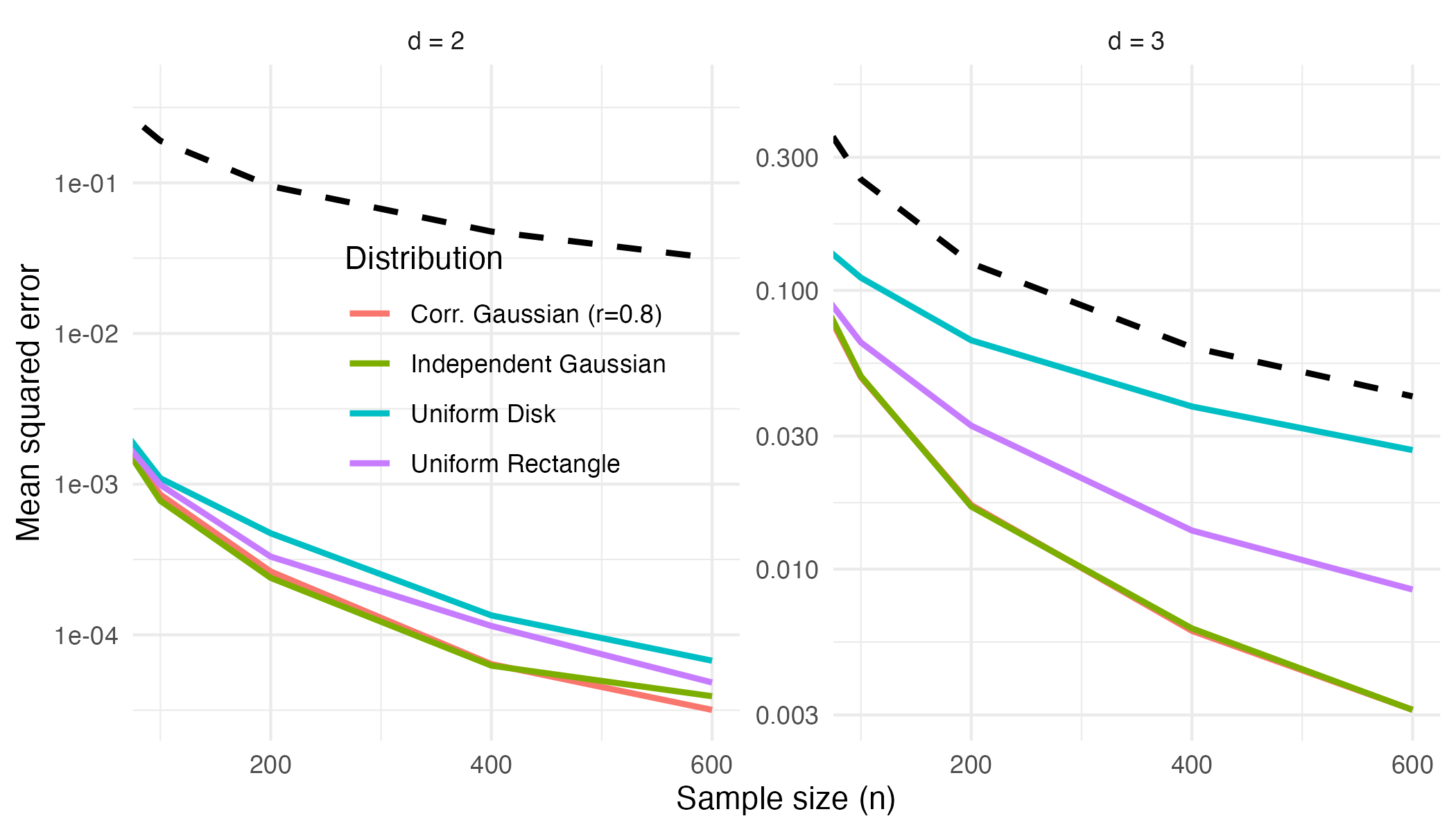}
    \caption{Verification of the MSE bound from Theorem \ref{convthm}: $\mathbb{E}[(V_n/n - D_{n})^2] \leq (6d+7)/n$. 
    The empirical mean squared error is shown for four distributions (uniform rectangle and disk, independent Gaussian, and correlated Gaussian with $r=0.8$) 
    in dimensions $d=2$ and $d=3$. 
    The black dashed line represents the theoretical upper bound. 
    All empirical values fall well below the bound across different probability measures. }
    \label{fig:main_theorem_mse}
\end{figure}

\begin{figure}[htbp]
    \centering
    \includegraphics[width=0.9\textwidth]{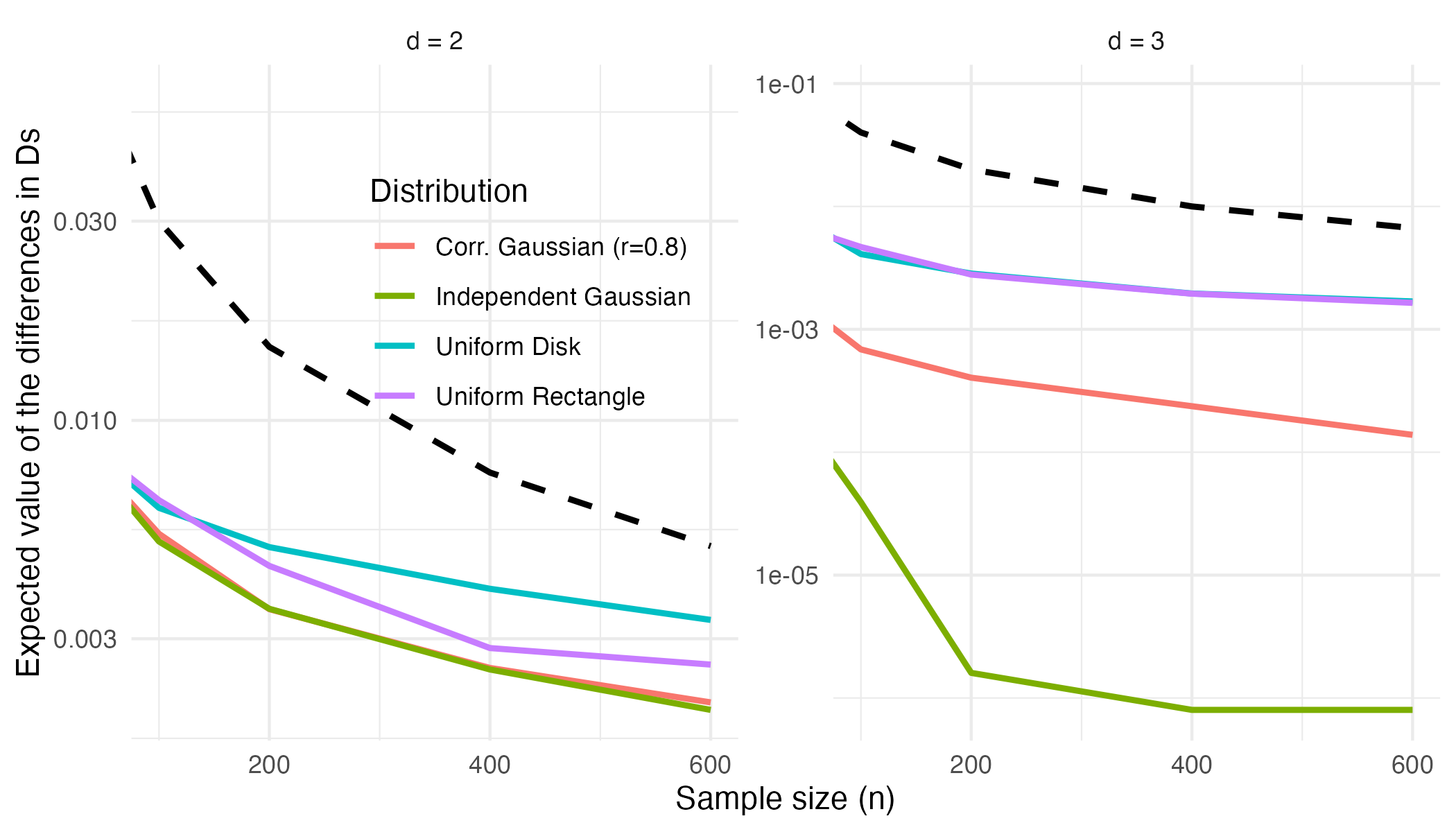}
    \caption{Verification of the second bound from Theorem \ref{convthm}: $\mathbb{E}[|D_n - D_{n-1}|] \leq (d+1)/n$. 
    The empirical expectation of the absolute difference between consecutive $D$ values is plotted against sample size for 
    the same four distributions and dimensions. 
    The black dashed line shows the theoretical bound $(d+1)/n$. The logarithmic scale emphasizes the $O(1/n)$ convergence rate predicted by the theory. }
    \label{fig:main_theorem_D_diff}
\end{figure}


\begin{figure}[htbp]
    \centering
    \includegraphics[width=0.95\textwidth]{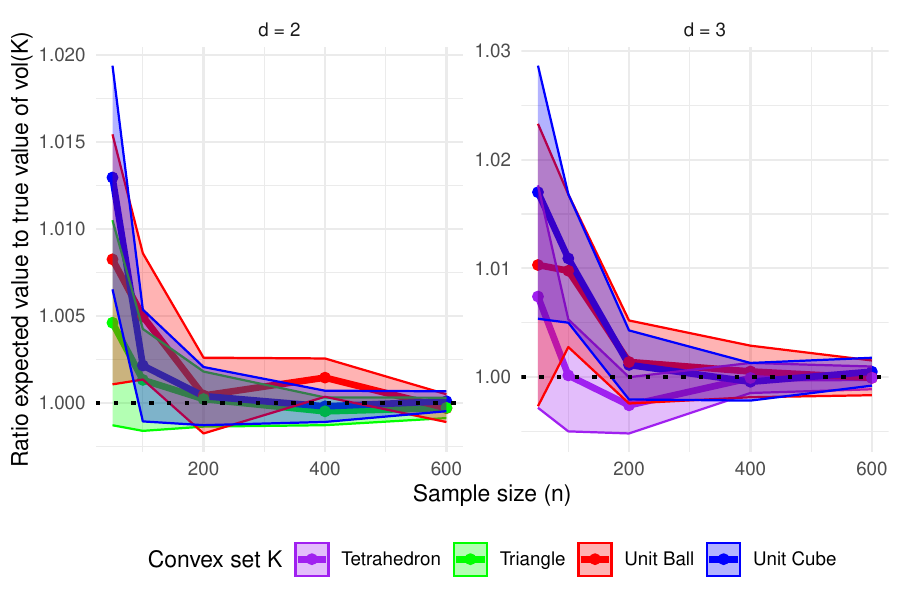}
    \caption{Illustrations of the volume estimator $\hat{\vol(K)}$ from  corollary \ref{convcor}. The ratio $\mathbb{E}[\hat{\vol(K)}/\vol(K)]$ approaches 1 (unbiased estimation) as sample size increases for unit cubes, unit balls, triangles, and tetrahedra in two and three dimensions. This convergence supports the validity of the volume estimation approach underlying the corollary's error bound.}
    \label{fig:correlarywithdifferentsupports}
\end{figure}

\subsubsection{Literature review}
The problem of estimating/approximating the volume of a convex set has generated a huge literature in the theoretical computer science community. Exact computation of the volume is one of the first problems shown to be \#P-complete (see \citet{valiant1979complexity} and \citet{dyer1988complexity}). An approximate computation of the volume in polynomial time is an early achievement of \citet{dyer1991random}. The first algorithms are of theoretical interest only (order $n^{24}$ or so). A long sequence of refinements reduced this to $O^*(n^5)$ (with omitted log factors) --- see \citet{kannan1997random}. The Wikipedia entry for `Convex volume approximation' \cite{wikipedia:convex_volume_approximation} has up-to-date references. To our knowledge, these estimates are still of only theoretical interest.

Inherent in these developments is a different type of `extra information', the idea of `reducible structure'. In our language this entails, in addition to $S$, a sequence $S=S_1\supseteq S_2\supseteq \cdots\supseteq S_n$ of sets with $|S_{i+1}|/|S_i|$ bound away from $0$ and $1$, and $|S_n|$ known. Random sampling from $S_i$ allows efficient estimation of $|S_{i+1}|/|S_i|$, and multiplying these estimates together and by the known $|S_n|$ gives an estimate of $|S|$. The original work of Broder, Jerrum and Sinclair \cite{sinclair1989approximate} makes this all precise. See \citet{diaconis1995three} and \citet{diaconis2025counting} for real examples. It is a challenging question to combine these reducible structure ideas with the present approach.

Return now to the present Theorem \ref{convthm} and Corollary \ref{convcor}. It is natural to `blow up' the volume $V_n$ of the convex hull of the sample along the lines of the German tank paradigm. In unpublished work, Diaconis suggested the volume estimate displayed in equation \eqref{volest}.  This suggestion was published and developed in two dimensions by \citet{Lo1992}, who gives a central limit approximation for the distribution.  \citet{baldin2016unbiased} discuss careful choice of the blow-up factor and prove that the resulting estimator is minimax.

\subsection{Interpolating between birthdays and German tanks using posets}
\label{sec:uppersets}
In the introduction, two variants of the problem `Estimate $|S|$' are discussed. In the first (birthday variant), no structure on $S$ is assumed and a sample size $\gg \sqrt{|S|}$ is required. In the second, $S$ is given with a linear order as $S = \{1,2,\ldots,|S|\}$ and a sample of any growing size suffices. This section interpolates between these when $S$ is partially ordered. The ideas are `easy' but since `poset theory' is not standard fare, we proceed slowly. Richard Stanley's \cite[Chapter 3]{stanley2011enumerative} is a standard reference to partially ordered sets. William Trotter's  \emph{Combinatorics and Partially Ordered Sets: Dimension Theory}~\cite{trotter92} is also useful.

A chain in a partially ordered set is a linearly ordered subset. An antichain has no elements comparable. (See Figure \ref{posetfig} for some examples.) The German tank problem has $S$ a chain, the birthday problem has $S$ an antichain. Most posets are somewhere between these two.
Our development begins abstractly; following this the birthday and German tank problem are treated; our estimates specialize exactly to the naive estimates \eqref{eq1.1}, \eqref{eq1.2} in the introduction.
Convex subsets of a poset are treated in section \ref{sec:convsubsets}.

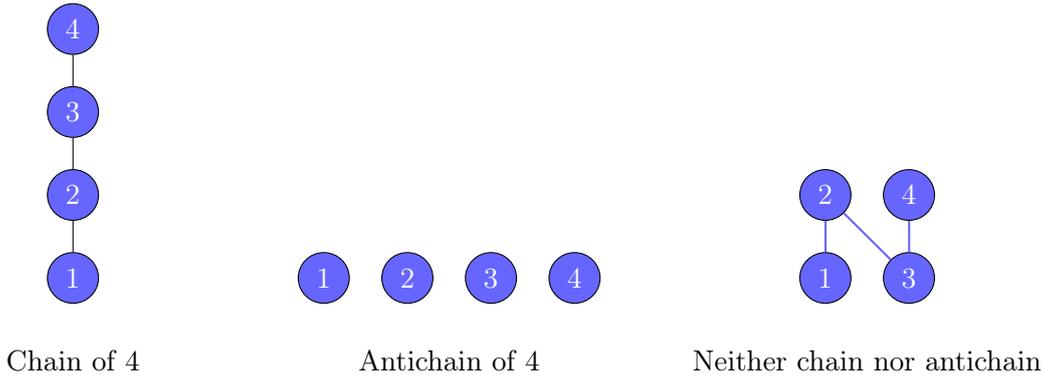
\begin{figure}
\begin{center}
\begin{tikzpicture}[scale=1.1]
    \node[draw, circle, fill=blue!60, text=white] (1) at (0,0) {1};
    \node[draw, circle, fill=blue!60, text=white] (2) at (0,1) {2};
    \node[draw, circle, fill=blue!60, text=white] (3) at (0,2) {3};
    \node[draw, circle, fill=blue!60, text=white] (4) at (0,3) {4};
    
    \draw (1) -- (2) -- (3) -- (4);
    
    \node at (0,-1) {Chain of 4};
    
    \node[draw, circle, fill=blue!60, text=white] (a1) at (3,0) {1};
    \node[draw, circle, fill=blue!60, text=white] (a2) at (4,0) {2};
    \node[draw, circle, fill=blue!60, text=white] (a3) at (5,0) {3};
    \node[draw, circle, fill=blue!60, text=white] (a4) at (6,0) {4};
    
    \node at (4.5,-1) {Antichain of 4};
    
    \node[draw, circle, fill=blue!60, text=white] (n1) at (10,0) {3};
    \node[draw, circle, fill=blue!60, text=white] (n2) at (9,0) {1};
    \node[draw, circle, fill=blue!60, text=white] (n3) at (9,1) {2};
    \node[draw, circle, fill=blue!60, text=white] (n4) at (10,1) {4};
    
    \draw[blue!60, thick] (n1) -- (n3) -- (n2);
    \draw[blue!60, thick] (n1) -- (n4);
        
    \node at (9.5,-1) {Neither chain nor antichain};
\end{tikzpicture}
\caption{Three examples of partially ordered sets. A line connecting two vertices indicates that they are comparable,  with the vertex positioned higher in the picture being greater than the lower in the partial ordering.\label{posetfig}}
\end{center}
\end{figure}

Let $(S,\preceq)$ be a partially ordered set, possibly infinite. Let $\mathcal{S}$ be a $\sigma$-algebra on $S$ that is compatible with the partial order, meaning that the set $\{(x,y)\in S^2: x\preceq y\}$ is a measurable subset of $S^2$ under the product $\sigma$-algebra. Recall that the upper set of an element $x\in S$ is the set of all $y$ such that $x\preceq y$. The upper set of a subset $T\subseteq S$ is the union of the upper sets of all $x\in T$. The upper set of $T$ is denoted by $\uparrow \! T$.

An up-set is simply a subset containing all larger points. For the 
poset 
\begin{center}
\tikz[baseline, scale=0.8]{
\node[draw, circle, fill=blue!60, text=white, inner sep=1pt] (n1) at (0.5,0) {3};
\node[draw, circle, fill=blue!60, text=white, inner sep=1pt] (n2) at (-0.5,0) {1};
\node[draw, circle, fill=blue!60, text=white, inner sep=1pt] (n3) at (-0.5,1) {2};
\node[draw, circle, fill=blue!60, text=white, inner sep=1pt] (n4) at (0.5,1) {4};
    
    \draw[blue!60, thick] (n1) -- (n3) -- (n2);
    \draw[blue!60, thick] (n1) -- (n4);
    }
\end{center}
the up-set generated by 3 is 
\begin{center}
    \tikz[baseline, scale=0.8]{
  \node[draw, circle, fill=blue!60, text=white, inner sep=1pt] (u3) at (0,0) {3};
  \node[draw, circle, fill=blue!60, text=white, inner sep=1pt] (u2) at (-0.5,0.5) {2};
  \node[draw, circle, fill=blue!60, text=white, inner sep=1pt] (u4) at (0.5,0.5) {4};
  \draw[blue!60, thick] (u3) -- (u2);
  \draw[blue!60, thick] (u3) -- (u4);
}.
\end{center}
Up-sets are a 
basic construct of poset theory where they are called `order filters' (see \citet[Section 3.4]{stanley2011enumerative}).
A simple example of an up-set comes from using the opposite order on $S = \{1,\ldots,N\}$ and the product order on $S\times S$; the up-sets are partitions
$\{(i,j), 1\leq i\leq I, 1\leq j \leq \lambda_i \mbox{ with } \lambda_1 \geq \lambda_2\geq \cdots \geq \lambda_I\}$.

Let $X_1,\ldots,X_n$ be i.i.d.~points drawn from a probability measure $\mu$ on $(S,\mathcal{S})$. Suppose that we want to estimate the measure of the upper set of $X_1,\ldots,X_n$ from this data. The following theorem, which we obtain as a consequence of Theorem \ref{mainthm}, shows that a good estimate can be produced without any extra knowledge about $S$ or $\mu$.
\begin{thm}\label{upthm}
Let $(S,\preceq)$ be a partially ordered set and $\mathcal{S}$ be a $\sigma$-algebra on $S$ that is compatible with the partial order, in the above sense. Let $X_1,\ldots,X_n$ be i.i.d.~points drawn from a probability measure $\mu$ on $(S,\mathcal{S})$. Let 
\[
N_n := |\{1\le i\le n: X_j \preceq X_i \textup{ for some } j\ne i\}|. 
\]
Then 
\[
\E\biggl[\biggl(\frac{N_n}{n} - \mu(\uparrow\!\{X_1,\ldots,X_n\})\biggr)^2\biggr] \le\biggl(\frac{6}{e} + \frac{1}{2}\biggr) \frac{1}{n} < \frac{3}{n}. 
\]
\end{thm}
\begin{proof}
To put this problem in the framework of Theorem \ref{mainthm}, let 
\[
A(x_1,\ldots,x_n) :=\; \uparrow\!\{x_1,\ldots,x_n\},
\]
and define $A'$ and $A''$ to be the same, but with $n-1$ and $n-2$ points, respectively. Then note that 
\begin{align*}
\mu(A(X_1,\ldots,X_n)) &= \mu(\uparrow \!\{X_1,\ldots,X_n\}), 
\end{align*}
and 
\begin{align*}
\frac{1}{n}\sum_{i=1}^n 1_{\{X_i \in A'(X_1,\ldots,X_{i-1},X_{i+1},\ldots,X_n)\}} &= \frac{1}{n}\sum_{i=1}^n 1_{\{X_i \in\, \uparrow\!\{X_1,\ldots,X_{i-1},X_{i+1},\ldots,X_n\}\}}\\
&= \frac{N_n}{n}.
\end{align*}
Thus, we only have to compute upper bounds on the quantities $\theta$, $\delta'$ and $\delta''$ from Theorem~\ref{mainthm}. We bound $\theta$ by $\frac{1}{4}$.  Next, note that
\begin{align*}
\delta'' &= \E[\mu(A'(X_1,\ldots,X_{n-1})\Delta A''(X_1,\ldots,X_{n-2}))]\\
&= \P(X_1\not \preceq X_n, \ldots, X_{n-2}\not\preceq X_n, X_{n-1}\preceq X_n).
\end{align*}
Let $Z := \P(X_1\not\preceq X_n|X_n)$. Then note that $Z = \P(X_i\not\preceq X_n|X_n)$ a.s.~for each $i\le n-1$, and 
\[
\P(X_{n-1}\preceq X_n|X_n) = 1- \P(X_{n-1}\not\preceq X_n|X_n) = 1-Z.
\]
Moreover, the events $X_i\not\preceq X_n$, $i=1,\ldots,n-1$, are conditionally independent given $X_n$. Thus, we get
\begin{align*}
&\P(X_1\not \preceq X_n, \ldots, X_{n-2}\not\preceq X_n, X_{n-1}\preceq X_n)\\
&= \E[\P(X_1\not \preceq X_n, \ldots, X_{n-2}\not\preceq X_n, X_{n-1}\preceq X_n |X_n)]\\
&= \E[\P(X_1\not \preceq X_n|X_n) \cdots \P(X_{n-2}\not\preceq X_n|X_n)\P(X_{n-1}\preceq X_n|X_n)]\\
&= \E[Z^{n-2}(1-Z)].
\end{align*}
But $Z$ takes value in $[0,1]$. Thus, simple calculus shows that the maximum possible value of $Z^{n-2}(1-Z)$ is attained when $Z = 1-\frac{1}{n-1}$, and this value is 
\[
\biggl(1-\frac{1}{n-1}\biggr)^{n-1} \frac{1}{n-1}\le \frac{1}{e(n-1)}.
\]
This proves that 
\[
\delta'' \le \frac{1}{e(n-1)}.
\]
By exactly the same argument with $n$ replaced by $n+1$ (and introducing an extra variable $X_{n+1}$), we get
\[
\delta'\le \frac{1}{en}.
\]
By Theorem \ref{mainthm}, this completes the proof.
\end{proof}

Theorem~\ref{upthm} can be used to estimate the size of a finite up-set from a sample of uniformly chosen i.i.d.~random points from that up-set, as follows. Let $T$ be a finite up-set, and let $X_1,\ldots,X_n$ be drawn independently and uniformly at random from $T$. Let $N_n$ be as in Theorem \ref{upthm}. Then by Theorem \ref{upthm}, we know that with high probability,
\[
\frac{N_n}{n} \approx \frac{|\!\uparrow\!\{X_1,\ldots,X_n\}|}{|T|}. 
\]
Thus, it is reasonable to estimate $|T|$ using 
\[
\hat{|T|} := \frac{n|\!\uparrow\!\{X_1,\ldots,X_n\}|}{N_n}.
\]
To get an upper bound on the error of this estimate, note that
\begin{align*}
\biggl(\frac{N_n}{n} - \frac{|\!\uparrow\!\{X_1,\ldots,X_n\}|}{|T|}\biggr)^2 &= \frac{N_n^2}{n^2}\biggl(1 - \frac{\hat{|T|}}{|T|}\biggr)^2.
\end{align*}
Thus, Theorem \ref{upthm} yields the following corollary.
\begin{cor}\label{upcor}
Let $S$ be a partially ordered set and $T$ be a finite up-set of $S$. Let $X_1,\ldots,X_n$ be drawn independently and uniformly from $T$, and let $N_n$ be as in Theorem~\ref{upthm}. Let $\hat{|T|}$ be the estimate of $|T|$ defined above. Then 
\[
\E\biggl[\frac{N_n^2}{n}\biggl(1 - \frac{\hat{|T|}}{|T|}\biggr)^2\biggr] \le \frac{6}{e}+\frac{1}{2}.
\]
\end{cor}
The above corollary implies that if  $N_n \gg \sqrt{n}$ in a particular realization, then we can expect that $\hat{|T|}$ is a good estimate of $|T|$. 
Using Markov's inequality as in the paragraph following Corollary \ref{convcor} gives confidence intervals.

Let us now work out some simple applications of Corollary \ref{upcor}.
\begin{ex}[Birthday generalization]
Let $S$ be an arbitrary set, and let $\preceq$ be the relation defined as $x\preceq x$ for each $x\in S$ and $x,y$ are not comparable if $x\ne y$. Let $T$ be any finite subset of $S$. Then $T$ is an up-set. Let $X_1,\ldots,X_n$ be drawn independently and uniformly at random from $T$. Estimating the size of $T$ from this sample is simply the inverse birthday problem. Let us work out what our estimator turns out to be in this example. Note that here, $N_n$ is the number of sample points that have been drawn more than once, and $\uparrow\!\{X_1,\ldots,X_n\}$ is just the sample set $\{X_1,\ldots,X_n\}$. Thus, in this example,
\begin{align}\label{hatt}
\hat{|T|} = \frac{n\cdot \text{number of distinct sample points}}{\text{number of sample points that are not singletons}}.
\end{align}
If sampling is done until a first repeat, then $\hat{|T|}=\frac{n(n-1)}{2}$. This is just the approximate maximum likelihood estimate \eqref{eq1.1} in the introduction. 

A different motivation for this same estimator follows from the `capture-recapture' work of
\citet{darroch1980note}. If $T$ is a finite set of size $|T|$ and $X_1,\ldots, X_n$ is an i.i.d.~uniform sample
from $T$, let $D$ be the set of distinct values among $\{ X_1,\ldots, X_n\}$. The chance that the next value is {\it not} in $D$ is $\frac{|T|-|D|}{|T|}$. Estimating this using the Good--Turing estimate 
$\frac{n_1}{n}$, with $n_1$ the number of singletons in the sample, \citet{darroch1980note}
equate $\frac{|T|-|D|}{|T|}\simeq \frac{n_1}{n}$, which gives the estimate
$$
\widehat{|T|}=\frac{|D|}{1-\frac{n_1}{n}}.
$$
Note that this is exactly the estimate displayed in equation \eqref{hatt}. If there are many repeats, this can be improved by the Chao estimates~\cite{chao1984nonparametric,chao1992estimating} that use doubletons as well as singletons.

Consider further the denser case  where we draw $n = \alpha |T|$ samples, where $\alpha$ is some given positive fraction. This is like dropping $n$ balls into $|T|$ boxes; the number of distinct sample points correspond to the number of nonempty boxes, which is $\approx |T|(1 - e^{-\alpha})$, and the number of sample points that are singletons correspond to the number of boxes containing exactly one ball, which is $\approx |T|\alpha e^{-\alpha}$. Thus, the number of sample points that are not singletons is $\approx n - |T|\alpha e^{-\alpha} = \alpha |T|(1-e^{-\alpha})$. Therefore,
\[
\hat{|T|} \approx \frac{n |T|(1-e^{-\alpha})}{\alpha |T|(1-e^{-\alpha})} = \frac{n}{\alpha}=|T|.
\]
Note the inverse birthday problem referred to here is different from the problem studied in \citet{hwangetal17} who
study the problem with unknown $n$.
\end{ex}

\begin{ex}[German tanks]
\label{extanks}
Let $S= \{1,2,\ldots\}$ be the set of natural numbers, and let $\preceq$ be the opposite of the usual ordering, that is, let $x\preceq y$ if and only if $y\le x$. Let $T$ be a finite up-set of $S$. Clearly, $T$ must be of the form $\{1,2,\ldots,|T|\}$. If $X_1,\ldots,X_n$ is an i.i.d.~uniform sample from $T$, then $\uparrow\!\{X_1,\ldots,X_n\}$ is the set $\{1,2,\ldots, R\}$ where $R := \max_{1\le i\le n} X_i$, and 
\[
N_n = 
\begin{cases}
n -1 &\text{ if the sample maximum is unique},\\
n &\text{ if not.}
\end{cases}
\]
Thus, we get 
\[
\hat{|T|} = 
\begin{cases}
\frac{n}{n -1}R &\text{ if the sample maximum is unique},\\
R &\text{ if not.}
\end{cases}
\]
It is easy to see from this expression that $\hat{|T|}$ is a good estimator of $|T|$ and is indeed the classical estimator (\ref{eq1.2}) for the tank problem.

What we have called the `German tank problem' was studied earlier by Harold Jeffreys as the tramcar problem \cite[section 4.8]{jeffreys1939theory},  for which he used a Bayesian method. Indeed, all
problems stated here can be given a Bayesian treatment. We plan to illustrate this in forthcoming work.

A most useful history of the problem of estimating $N$ based on a uniform sample size $n$ appears in a paper of Spencer and Langley \cite{spencer1993geary}. Their focus is
on the statistician R.~C.~Geary's work on the problem, but along the way they give a  host of earlier references, tracing the problem back to C.~S.~Peirce and Laplace \cite[p. 286, footnotes 2,3]{spencer1993geary}. Geary suggested the interesting estimate $\widehat{N}=  2^{\frac{1}{n}}\max_{1\le i\le n} x_i$ and showed that it was `closer' to $N$ than the MLE in the sense of Pitman. 
\end{ex}

\begin{ex}[Subtrees]\label{treeex}
Let $S$ be the set of nodes of an infinite rooted labeled tree. We will say that $x\preceq y$ if either $y=x$ or $y$ is an ancestor of $x$. Let $T$ be a finite up-set, which in this case means that $T$ is a finite subtree of $S$ that contains the root. Let $X_1,\ldots,X_n$ be an i.i.d.~uniform random  sample from $T$. Then $\uparrow\!\{X_1,\ldots,X_n\}$ is the subtree of $S$ consisting of all nodes that are either in the sample or has a descendant in the sample, and $N_n$ is the number of sample points which have neither any duplicates nor any descendants in the sample. Then Corollary \ref{upcor} says that 
\[
\hat{|T|} = \frac{n|\!\uparrow\!\{X_1,\ldots,X_n\}|}{N_n}
\]
is a good estimator of $|T|$ when $n$ is large and $N_n$ turns out to be $\gg \sqrt{n}$. It is not clear whether there is a simple way to see this, especially when nothing is known about the structure of the subtree $T$. \citet[Exercise 3.74]{stanley2011enumerative} shows that the poset of trees is precisely the set of binary stopping rules. This seems to be worth developing.
\end{ex}

We have not seen previous literature on sampling in the presence of a partial order. Indeed since the poset itself is an up-set, theorem \ref{upthm} gives a novel estimate for the size of a poset. A worthwhile potential application is Dedekind's problem of estimating the number of order ideals in the Boolean lattice, see \cite{HarperB} for references and first efforts.

However, there is extensive literature on computing for posets. For example, a basic theorem of Dilworth says that a finite poset can be covered by a disjoint union of $n$ chains if and only if the size of the largest antichain is $n$. The literature on efficient computation of $n$ and such covers can be found in \cite{badr2023efficient}.
One can think about `how to estimate the size of a poset' given a decomposition into disjoint chains as alternatives.

\subsection{Convex subsets of posets}
\label{sec:convsubsets}
Let $(S, \mathcal{S})$ be a poset together with a compatible $\sigma$-algebra, as in the previous subsection. The {\it convex hull} of a subset $A\subseteq S$, denoted by $\conv(A)$, is the set of all $x\in S$ that are sandwiched between two elements of $y,z \in A$, meaning that $y\preceq x\preceq z$. 

Let $X_1,\ldots,X_n$ be i.i.d.~points drawn from a probability measure $\mu$ on $(S,\mathcal{S})$. Suppose that we want to estimate the measure of $\conv(X_1,\ldots,X_n)$ from this data. The following theorem, which we obtain as a consequence of Theorem \ref{mainthm}, shows that a good estimate can be produced without any extra knowledge about $S$ or $\mu$.
\begin{thm}\label{posetconvthm}
Let $(S,\preceq)$ be a partially ordered set and $\mathcal{S}$ be a $\sigma$-algebra on $S$ that is compatible with the partial order. Let $X_1,\ldots,X_n$ be i.i.d.~points drawn from a probability measure $\mu$ on $(S,\mathcal{S})$. Let 
\[
N_n := |\{1\le i\le n: X_j \preceq X_i \preceq X_k \textup{ for some } j, k\ne i\}|. 
\]
Then 
\[
\E\biggl[\biggl(\frac{N_n}{n} - \mu(\conv(X_1,\ldots,X_n))\biggr)^2\biggr] \le\biggl(\frac{12}{e} + \frac{1}{2}\biggr) \frac{1}{n} < \frac{5}{n}. 
\]
\end{thm}
\begin{proof}
To put this problem in the framework of Theorem \ref{mainthm}, let 
\[
A(x_1,\ldots,x_n) := \conv(x_1,\ldots,x_n),
\]
and define $A'$ and $A''$ to be the same, but with $n-1$ and $n-2$ points, respectively. Then note that 
\begin{align*}
\mu(A(X_1,\ldots,X_n)) &= \mu(\conv(X_1,\ldots,X_n)), 
\end{align*}
and 
\begin{align*}
\frac{1}{n}\sum_{i=1}^n 1_{\{X_i \in A'(X_1,\ldots,X_{i-1},X_{i+1},\ldots,X_n)\}} &= \frac{1}{n}\sum_{i=1}^n 1_{\{X_i \in\conv(X_1,\ldots,X_{i-1},X_{i+1},\ldots,X_n)\}}\\
&= \frac{N_n}{n}.
\end{align*}
Thus, we only have to compute upper bounds on the quantities $\theta$, $\delta'$ and $\delta''$ from Theorem~\ref{mainthm}. We bound $\theta$ by $\frac{1}{4}$.  Next, note that if $X_n \in A'(X_1,\ldots,X_{n-1})\Delta A''(X_1,\ldots,X_{n-2})$, then $X_n\in \conv(X_1,\ldots,X_{n-1})$ and $X_n \notin \conv(X_1,\ldots,X_{n-2})$. This implies that either $X_n \preceq X_{n-1}$ and $X_n \not \preceq X_j$ for all $j\le n-2$, or $X_{n-1}\preceq X_n$ and $X_j \not \preceq X_n$ for all $j\le n-2$. Thus,
\begin{align*}
\delta'' &= \E[\mu(A'(X_1,\ldots,X_{n-1})\Delta A''(X_1,\ldots,X_{n-2}))]\\
&= \P(X_1\not \preceq X_n, \ldots, X_{n-2}\not\preceq X_n, X_{n-1}\preceq X_n)\\
&\qquad + \P(X_n\not \preceq X_1, \ldots, X_{n}\not\preceq X_{n-2}, X_{n}\preceq X_{n-1}).
\end{align*}
Let $Z := \P(X_1\not\preceq X_n|X_n)$. Then as in the proof of Theorem \ref{upthm}, we have
\begin{align*}
\P(X_1\not \preceq X_n, \ldots, X_{n-2}\not\preceq X_n, X_{n-1}\preceq X_n) &= \E[Z^{n-2}(1-Z)]\le \frac{1}{e(n-1)}.
\end{align*}
By a similar argument,
\[
\P(X_n\not \preceq X_1, \ldots, X_{n}\not\preceq X_{n-2}, X_{n}\preceq X_{n-1})\le \frac{1}{e(n-1)}.
\]
This proves that 
\[
\delta'' \le \frac{2}{e(n-1)}.
\]
By the same argument with $n$ replaced by $n+1$, we get
\[
\delta'\le \frac{2}{en}.
\]
By Theorem \ref{mainthm}, this completes the proof.
\end{proof}

A subset $T \subseteq S$ is called convex if $\conv(T) = T$. 
An interval $\left[a,b\right]=\{c: a\preceq  c \preceq  b\}$ is convex.
Referring to Example \ref{extanks}, the convex subsets of $S\times S$ are the skew-partitions \citet[section 1.10]{stanley2011enumerative}.

Theorem~\ref{posetconvthm} can be used to estimate the size of a finite convex subset from a sample of uniformly chosen i.i.d.~random points from that set, as follows. Let $T$ be a finite convex subset of $S$, and let $X_1,\ldots,X_n$ be drawn independently and uniformly at random from $T$. Let $N_n$ be as in Theorem \ref{posetconvthm}. Then by Theorem \ref{posetconvthm}, we know that with high probability,
\[
\frac{N_n}{n} \approx \frac{|\conv(X_1,\ldots,X_n)|}{|T|}. 
\]
Thus, it is reasonable to estimate $|T|$ using 
\[
\hat{|T|} := \frac{n|\conv(X_1,\ldots,X_n)|}{N_n}.
\]
To get an upper bound on the error of this estimate, note that
\begin{align*}
\biggl(\frac{N_n}{n} - \frac{|\conv(X_1,\ldots,X_n)|}{|T|}\biggr)^2 &= \frac{N_n^2}{n^2}\biggl(1 - \frac{\hat{|T|}}{|T|}\biggr)^2.
\end{align*}
Thus, Theorem \ref{posetconvthm} yields the following corollary.
\begin{cor}\label{posetconvcor}
Let $S$ be a partially ordered set and $T$ be a finite convex subset of $S$. Let $X_1,\ldots,X_n$ be drawn independently and uniformly from $T$, and let $N_n$ be as in Theorem~\ref{upthm}. Let $\hat{|T|}$ be the estimate of $|T|$ defined above. Then 
\[
\E\biggl[\frac{N_n^2}{n}\biggl(1 - \frac{\hat{|T|}}{|T|}\biggr)^2\biggr] \le \frac{12}{e}+\frac{1}{2}< 5.
\]
\end{cor}
The above corollary implies that if  $N_n \gg \sqrt{n}$ in a particular realization, then we can expect that $\hat{|T|}$ is a good estimate of $|T|$. 

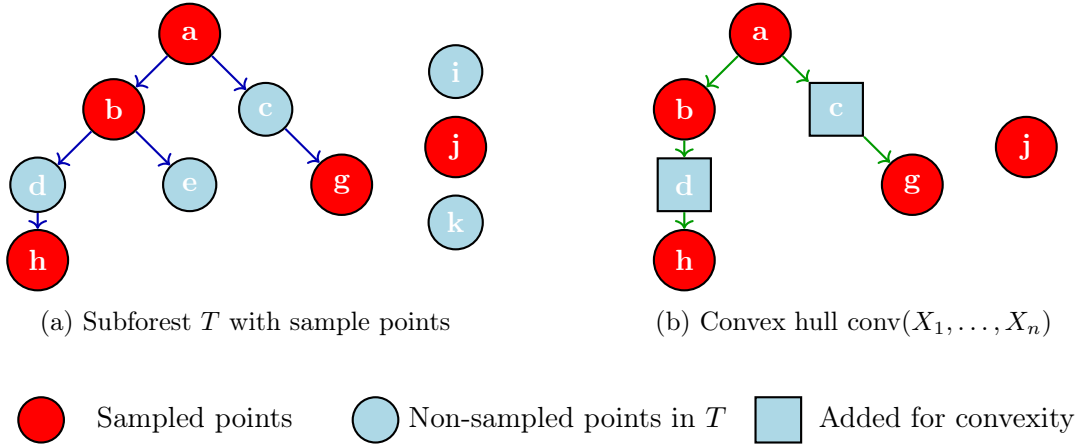
\begin{figure}[htbp]
    \centering
    \begin{subfigure}[b]{0.48\textwidth}
        \centering
        \begin{tikzpicture}[
            sampled/.style={circle, draw=black, fill=red, thick, minimum size=8mm},
            not_sampled/.style={circle, draw=black, fill=lightblue, thick, minimum size=7mm},
            arrow_style/.style={->, thick, blue!70!black},
            node_label/.style={font=\bfseries, text=white}
        ]
        
        \node[sampled, node_label] (a) at (0, 4) {a};
        \node[sampled, node_label] (b) at (-1, 3) {b};
        \node[not_sampled, node_label] (c) at (1, 3) {c};
        \node[not_sampled, node_label] (d) at (-2, 2) {d};
        \node[not_sampled, node_label] (e) at (0, 2) {e};
        \node[sampled, node_label] (g) at (2, 2) {g};
        \node[sampled, node_label] (h) at (-2, 1) {h};
        
        \node[not_sampled, node_label] (i) at (3.5, 3.5) {i};
        \node[sampled, node_label] (j) at (3.5, 2.5) {j};
        \node[not_sampled, node_label] (k) at (3.5, 1.5) {k};
        
        \draw[arrow_style] (a) -- (b);
        \draw[arrow_style] (a) -- (c);
        \draw[arrow_style] (b) -- (d);
        \draw[arrow_style] (b) -- (e);
        \draw[arrow_style] (c) -- (g);
        \draw[arrow_style] (d) -- (h);
        
        \end{tikzpicture}
        \caption{Subforest $T$ with sample points}
        \label{fig:subforest}
    \end{subfigure}
    \hfill
    \begin{subfigure}[b]{0.48\textwidth}
        \centering
        \begin{tikzpicture}[
            sampled/.style={circle, draw=black, fill=red, thick, minimum size=8mm},
            added_conv/.style={rectangle, draw=black, fill=lightblue, thick, minimum size=7mm},
            arrow_style/.style={->, thick, green!60!black},
            node_label/.style={font=\bfseries, text=white}
        ]
        
        \node[sampled, node_label] (a) at (0, 4) {a};      
        \node[sampled, node_label] (b) at (-1, 3) {b};        
        \node[added_conv, node_label] (c) at (1, 3) {c};      
        \node[added_conv, node_label] (d) at (-1, 2) {d};        
        \node[sampled, node_label] (g) at (2, 2) {g};         
        \node[sampled, node_label] (h) at (-1, 1) {h};        
        
        \node[sampled, node_label] (j) at (3.5, 2.5) {j};
        
        \draw[arrow_style] (a) -- (b);
        \draw[arrow_style] (a) -- (c);
        \draw[arrow_style] (b) -- (d);
        \draw[arrow_style] (c) -- (g);
        \draw[arrow_style] (d) -- (h);
        
        \end{tikzpicture}
        \caption{Convex hull $\conv(X_1,\ldots,X_n)$}
        \label{fig:convex_hull_forest}
    \end{subfigure}
    
    \vspace{0.3cm}
    \begin{center}
    \begin{tikzpicture}
        \node[circle, draw=black, fill=red, thick, minimum size=6mm] at (-2, 0) {};
        \node[right] at (-1.4, 0) {Sampled points};
        \node[circle, draw=black, fill=lightblue, thick, minimum size=6mm] at (2.4, 0) {};
        \node[right] at (2.7, 0) {Non-sampled points in $T$};
        \node[rectangle, draw=black, fill=lightblue, thick, minimum size=6mm] at (7.7, 0) {};
        \node[right] at (8.1, 0) {Added for convexity};
    \end{tikzpicture}
    \end{center}
    
    \caption{Illustration of subforest sampling. (a) A finite convex subset $T$ (subforest) with sample points $a, b, g, h, j$ (red circles) and non-sampled points $c, d, e, i, k$ (blue circles). (b) The convex hull $\conv(X_1,\ldots,X_n)$ includes sampled points (red circles), points added to maintain convexity (blue squares: $c,d$), and isolated sampled point $j$. Points $e, i, k$ are excluded as they don't contribute to the convex hull.}
    \label{fig:subforest_example}
\end{figure}

\begin{ex}[Subforests]
Let $S$ be the set of nodes of an infinite rooted labeled tree. As in Example \ref{treeex}, we will say that $x\preceq y$ if either $y=x$ or $y$ is an ancestor of $x$. Let $T$ be a finite convex subset of $S$, which in this case just means that $T$ is a finite subforest (i.e., union of subtrees) of $S$, with the property that whenever $x,y\in T$ and $x$ is a descendant of $y$, all points in the path from $y$ to $x$ are also in $T$. A key difference with Example \ref{treeex} is that $T$ is no longer required to contain the root. Let $X_1,\ldots,X_n$ be an i.i.d.~uniform random  sample from~$T$. Then $\conv(X_1,\ldots,X_n)$ is the subforest of $S$ consisting of all nodes that are either in the sample or has both a descendant and an ancestor in the sample, and $N_n$ is the number of sample points which do not have duplicates in the sample and are not sandwiched between two other elements in the sample (that is, an ancestor and a descendant). See Figure \ref{fig:subforest_example} for an illustration. Then Corollary~\ref{posetconvcor} says that 
\[
\hat{|T|} = \frac{n|\conv(X_1,\ldots,X_n)|}{N_n}
\]
is a good estimator of $|T|$ when $n$ is large and $N_n$ turns out to be $\gg \sqrt{n}$. 
\end{ex}

Section \ref{sec:convexhull} and the present section both offer methods for
estimating the volume of a convex set. There are further abstractions of convexity
which might be amenable to the present approach. Perhaps most promising is the abstract-convexity-anti-matroid notions wonderfully exposed in \citet{korte2012greedoids}. The book by \citet{van1993theory} focuses on convexity via closure operators. This is applied to function approximation in \cite{millan2022application}.
\subsection{Testing coincidences}\label{aldoussec}
This section develops an idea of David Aldous presenting an abstract problem. We begin with Aldous's informal description
and then turn to a general theorem and follow this by an example
and a literature review.

\subsubsection{Suspicious coincidences (thanks to D.~Aldous)}
Suppose we have a large database of different objects of the same type, and we want to decide whether a new object is very similar to some object in the database --- more similar than could be expected `by chance'. Examples are fingerprints,
human DNA (in the forensic context),
facial recognition, musical tunes or lyrics in the copyright context, or even a plot of a new murder mystery.

Any quantitative decision method must involve some scheme (explicit or implicit)
for assessing a quantitative dissimilarity --- a distance --- between two objects, then
finding the object in the database that is closest to the new object, then considering
whether it is `too close to be just by chance', which would then suggest some causal
relationship. So a natural model in this general context is that there is a space $(S,d)$
of possible objects and distances, and that our database objects and the new object
are i.i.d.~samples from some probability measure $\mu$ on $S$. Suppose that we do not observe $S$ or $\mu$; 
all we observe are all the distances between these objects, which may be given
by some complicated algorithm (in the DNA or the facial recognition examples)
or by human judgment (the other examples).
This is a cleaned up mathematical setting where we observe only the $D_{i,j}$, and
we seek to
make inferences which are `universal' in the sense of not depending on $(S,\mu)$.

Now the decision problem we are considering actually has an `obvious' solution,
as follows. In our database of $n$ objects, for each object $i$ we can find the distance
$D_i$ to the closest other object. Because our new object will be drawn from the same
distribution as the database objects, then under the `by chance' hypothesis, the
distribution of the distance $D$ from the new object to the nearest database object
should be essentially the same as the empirical distribution of nearest-neighbor
distances ($D_i, 1 \leq i \leq n$). Theorem~\ref{universaldthm} formalizes this `essentially the same' idea,
in the desired `universal' way. So we obtain a classical-style `statistical hypothesis test' by simply comparing the observed distance $D$ with the empirical distribution of ($D_i, 1 \leq i \leq n$).

\subsubsection{A universal approximation theorem}

Let $(S,d)$ be a metric space endowed with its Borel $\sigma$-algebra, and let $X_1,\ldots,X_n$ be i.i.d.~$S$-valued random variables with law $\mu$. Let $B(x,r)$ denote the closed ball of radius $r$ and center $x$. Let
\[
U_n(r) := \bigcup_{i=1}^n B(X_i,r).
\]
Observe that 
\begin{align*}
\mu(U_n(r)) &= \text{Prob(The nearest-neighbor distance of a new draw from $\mu$ }\\
&\qquad \qquad \qquad \text{from the existing sample is $\le r$).}
\end{align*}
Suppose that we want to estimate $\mu(U_n(r))$ from the data without using any prior knowledge about $\mu$. The following theorem, obtained from Theorem \ref{mainthm}, gives an estimate whose error has no dependence on $\mu$ or $S$. (A version of this result was communicated by the first author to Andreas Maurer some years ago; it has appeared in the paper~\cite{maurer22}.)
\begin{thm}
\label{universaldthm}
Let $(S,d)$ be a metric space endowed with its Borel $\sigma$-algebra, and let $X_1,\ldots,X_n$ be i.i.d.~$S$-valued random variables with law $\mu$. Let $U_n(r)$ be defined as above. Let $W_n(r)$ be the number of $i$ such that the distance of $X_i$ to its nearest neighbor in $X_1,\ldots,X_{i-1}, X_{i+1},\ldots,X_n$ is $\le r$. Then 
\[
\E\biggl[\biggl(\frac{W_n(r)}{n}-\mu(U_n(r))\biggr)^2\biggr]\le \frac{7}{n}.
\]
\end{thm}
\begin{proof}
Fix some $r\ge 0$. To put this problem in the framework of Theorem \ref{mainthm}, let 
\[
A(x_1,\ldots,x_n) := \bigcup_{i=1}^n B(x_i,r),
\]
and define $A'$ and $A''$ to be the same, but with $n-1$ and $n-2$ points, respectively. Then note that 
\begin{align*}
\mu(A(X_1,\ldots,X_n)) &= \mu(U_n(r)), 
\end{align*}
and 
\begin{align*}
\frac{1}{n}\sum_{i=1}^n 1_{\{X_i \in A'(X_1,\ldots,X_{i-1},X_{i+1},\ldots,X_n)\}} &= \frac{1}{n}\sum_{i=1}^n 1_{\{\min_{j\ne i} d(X_i,X_j)\le  r\}}\\
&= \frac{W_n(r)}{n}.
\end{align*}
Thus, we only have to compute upper bounds on the quantities $\theta$, $\delta'$ and $\delta''$ from Theorem~\ref{mainthm}. We bound $\theta$ by $\frac{1}{4}$. Next, for $1\le i\le n$, define
\[
B_i := B(X_i,r) \setminus \bigcup_{1\le j\le n, \, j\ne i} B(X_j, r). 
\]
Then note that
\begin{align*}
\delta' &= \E[\mu(A(X_1,\ldots,X_{n})\Delta A'(X_1,\ldots,X_{n-1}))]= \E[\mu(B_n)].
\end{align*}
Since the sets $B_1,\ldots,B_n$ are disjoint, symmetry implies that
\begin{align*}
\E[\mu(B_n)] &= \frac{1}{n}\sum_{i=1}^n \E[\mu(B_i)]= \frac{1}{n} \E\biggl[\mu\biggl(\bigcup_{i=1}^n B_i\biggr)\biggr] \le \frac{1}{n}.
\end{align*}
Thus, we get that
\[
\delta'\le \frac{1}{n}.
\]
By exactly the same argument with $n$ replaced by $n-1$, we get
\[
\delta''\le \frac{1}{n-1}.
\]
By Theorem \ref{mainthm}, this completes the proof.
\end{proof}
{\it Remark:} Theorem \ref{universaldthm} gives $O(1/\sqrt{n})$ concentration. An unpublished
example of Aldous shows this cannot be improved without further assumptions. He conjectured that $$
\E\biggl(\sup_{r\ge 0}\biggl\{ \frac{W_n(r)}{n} -\mu(U_n(r)) \biggr\}^2\biggr) \leq \frac{c}{n},\mbox{ for some universal }c. 
$$
This is an open problem. 

In real applications, more information is often available. This is discussed below.

\subsubsection{Statistical discussion (real world cases)}
This section gives two examples of the use of Theorem \ref{universaldthm}, the first on simulated data that we use to create null distributions, the second using
actual DNA sequences from a standard database.
For the first, we 
created a population of
 $200$ DNA sequences of length $400$ with typical nucleotide frequencies.
Consider a sample of size $40$ from these $200$. In
an application this sample would be used with the leave one out procedure, computing the nearest neighbor distance from each of the $40$ points to the remaining $39$. These $40$ numbers would then be used to calibrate the minimum distance to a fresh point.
How accurate is this estimate? To evaluate this we create random splits of the $200$ points into 
 sets (of size $40$ and $160$) $500$ times. 

For each of these $500$ splits we computed the 
$40$ nearest neighbor leave one out distances of each sample point. We also computed  
the minimum distance of each of the $160$ other points to the chosen $40$ points. This gives $160$ distances --- the population-to-sample distances. For each of our
$500$ random splits we have two histograms.
Four of the $500$ are shown  in Figure~\ref{fig:example_dist}. A standard smoother was used to create the densities.

One way to compare each pair of histograms
 is to use the Anderson--Darling two sample statistics. The  histogram of the $500$ statistics is shown in Figure \ref{fig:ad_stats}.

\begin{figure}[htbp]
\centering
\begin{subfigure}[b]{0.49\textwidth}
    \centering
    \includegraphics[width=\textwidth]{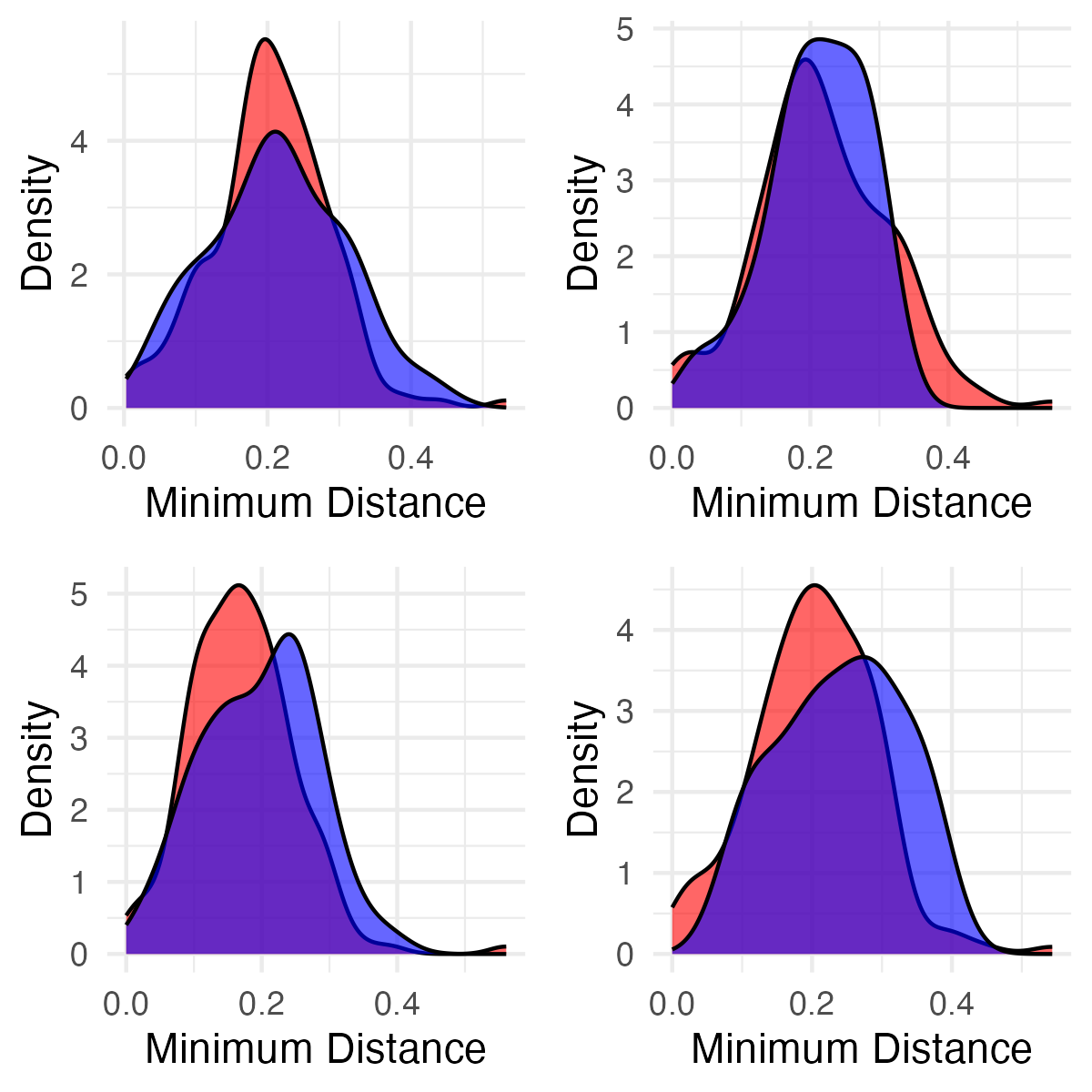}
    \caption{A few exemplary distribution comparisons}
    \label{fig:example_dist}
\end{subfigure}
\hfill
\begin{subfigure}[b]{0.45\textwidth}
    \centering
    \includegraphics[width=\textwidth]{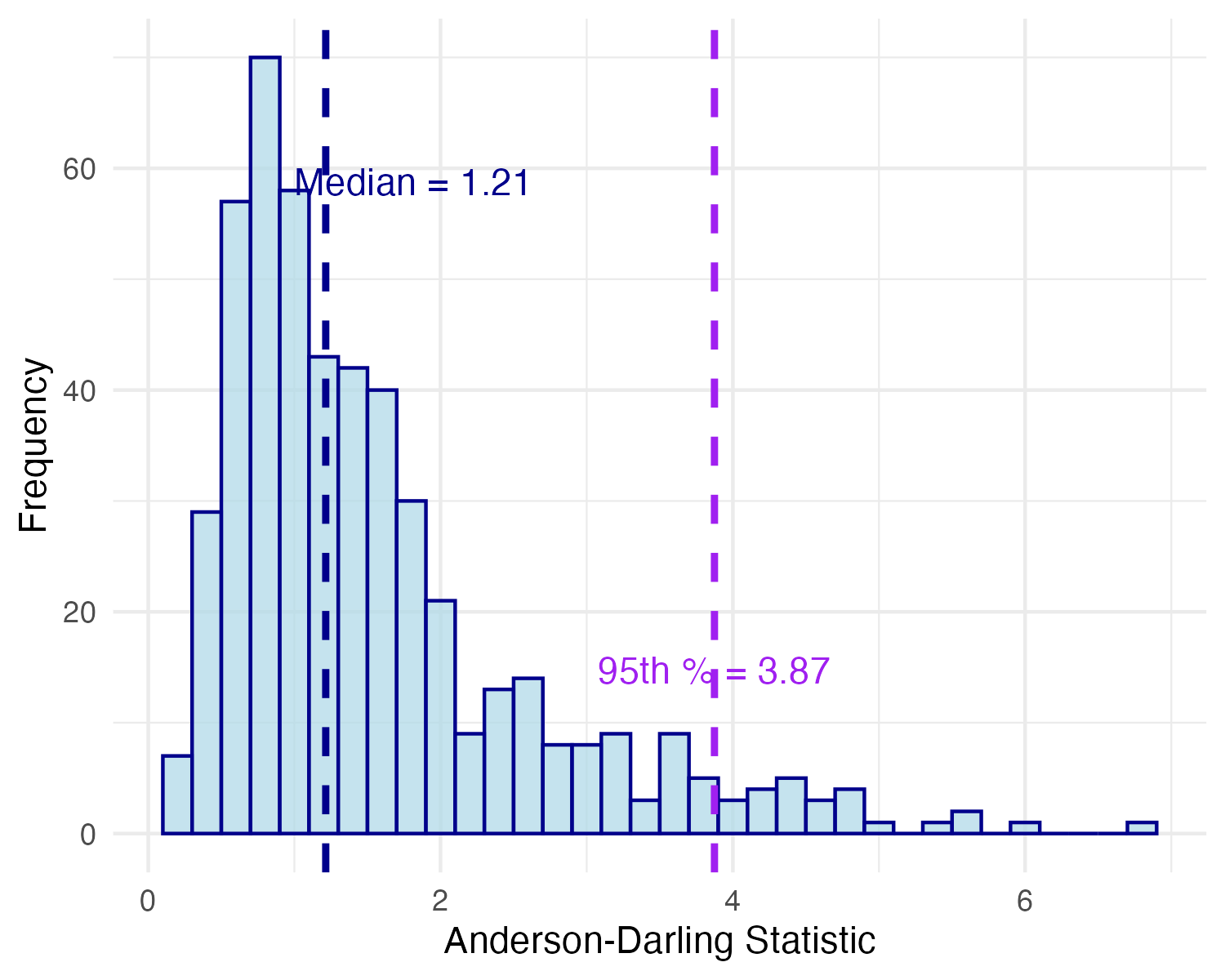}
    \caption{Anderson--Darling two sample statistics 500 Monte Carlo simulations.}
    \label{fig:ad_stats}
\end{subfigure}
\caption{Comparison of Within-Sample and Sample-to-Population Distance Distributions. (a) Four example simulations showing density plots of nearest neighbor distances. (b) Histogram of the Anderson--Darling test statistics across 500 simulations in blue.}
\label{fig:distance_comparisons}
\end{figure}

We formed a ``null'' distribution for the Anderson--Darling statistic by simulating
 $1000$  DNA sequences of length $400$. 
We then:
\begin{itemize}
\item Picked $40$ points from an overall (larger) population of $1000$ at random and computed their minimum distances to $40$ fresh random points from the population. This gives a first sample of $40$ numbers. 
\item Pick $160$ points from the population and computed the minimum distance for each of these to another $40$ randomly chosen points. This gives a second sample of $160$ numbers.
\item Computed the Anderson--Darling two sample statistic for these two samples.
\end{itemize}
These three steps were repeated $500$ times 
to make up a null distribution for the  Anderson--Darling $160$-$40$ statistic. This null distribution is compared to
the results from Figure \ref{fig:ad_stats} 
in Figure \ref{fig:adnull}.
We see that the leave one out distributions are in good agreement with the null distribution.

\begin{figure}[htbp]
    \centering
    \begin{subfigure}[b]{0.48\textwidth}
        \centering
\includegraphics[width=\textwidth,height=0.8\textwidth]{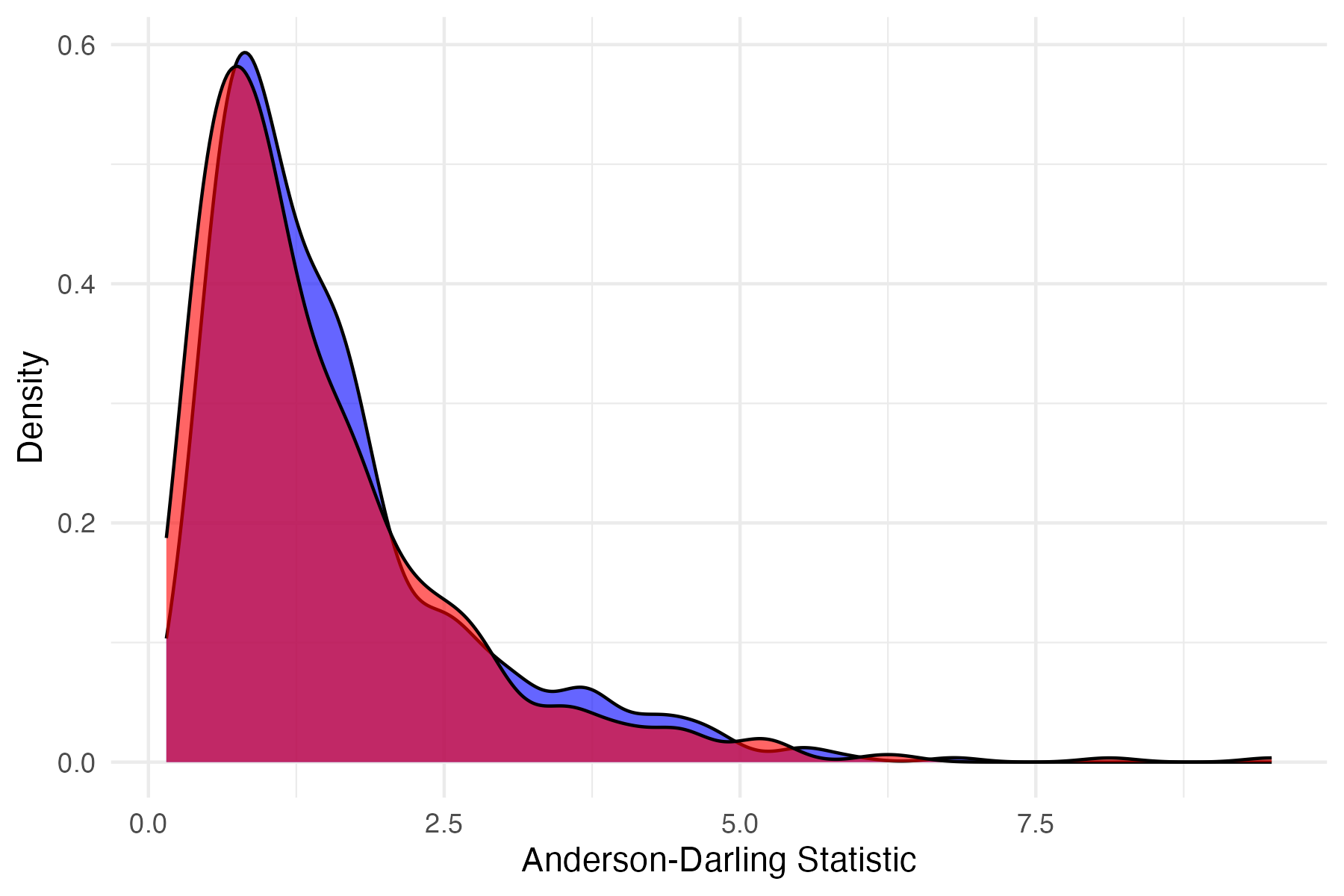}
\caption{Anderson--Darling observed versus null. The null is in red, the blue is the leave one out.}
     \label{fig:adnull}   
    \end{subfigure}
    \hfill
    \begin{subfigure}[b]{0.45\textwidth}
        \centering
\includegraphics[width=\textwidth,height=0.8\textwidth]{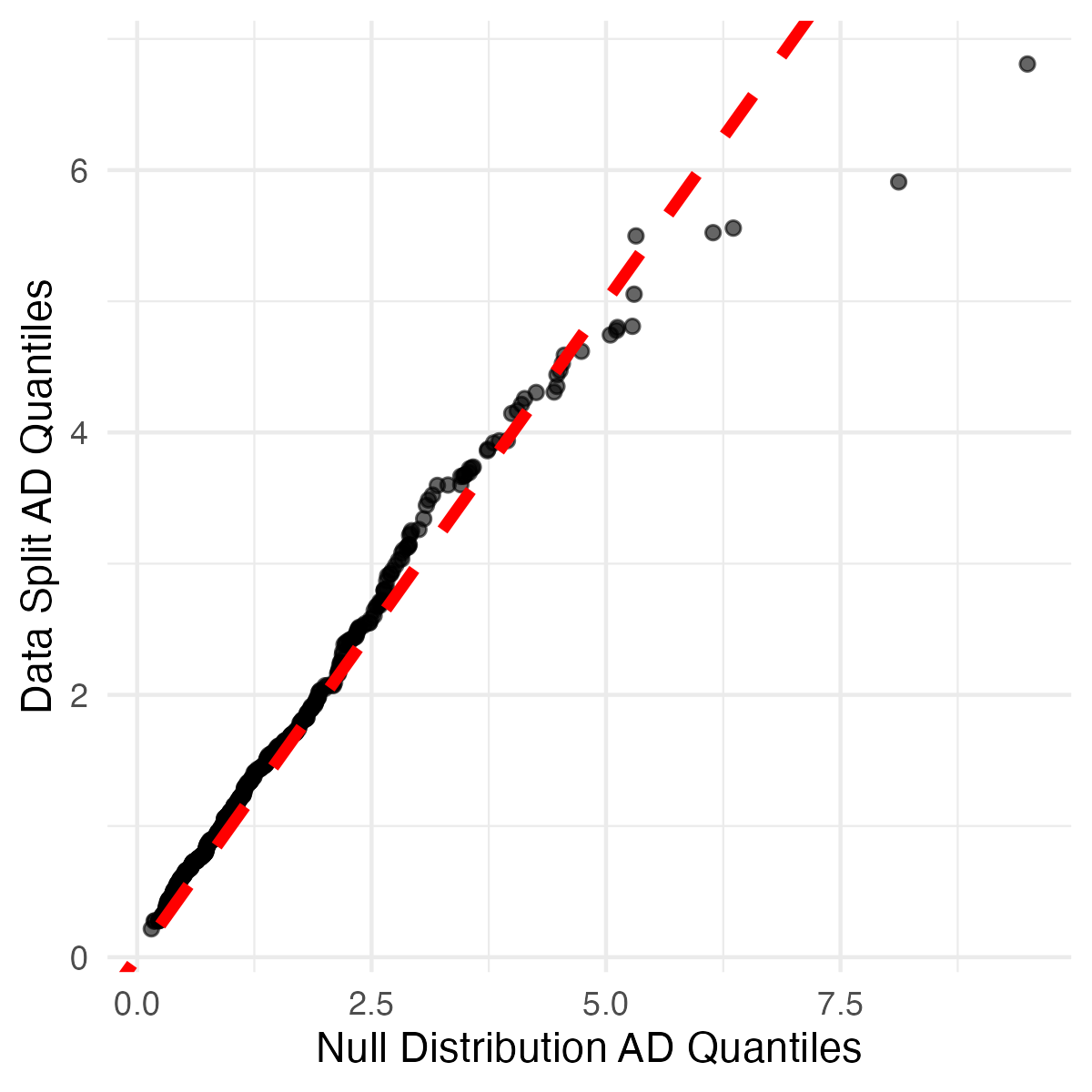}
        \caption{QQ plot of the Anderson Darling statistics.\\ }
    \end{subfigure}
    \caption{Comparison of Anderson--Darling statistics. (a) The AD statistic for the leave one out to population (in blue) comparison superimposed with the null distribution (in red)  for the Anderson--Darling statistic. (b) The QQ plot of the quantiles of the AD statistic computed 500 times from random 160-40 splits of our observed data compared to the true null of the AD statistic for 160-40 samples from a large population.}
    \label{fig:ad_analysis}
\end{figure}

For the second example, standard genomic databases can be used to produce the distances and provide approximations
of $\mu$. For instance, in denoising microbiome data, the DADA2~\cite{callahan2016dada2} algorithm proceeds by using the frequencies of the given data
set in bins to approximate the probability that a new
sequence is a noisy version of an already encountered sequence
or whether it is a new entity. This is the same idea as when Google suggests that you have made a typo `did you mean banana?' because it has the data and the frequency of `banana' is so much
higher than that of the one you typed in (`bannana').
It is important in such cases to use the whole distribution of distances from the data because using a fixed radius of say 99\% similarity is invalid, since it does not take into account the baseline 
data density.

As an  example, that is developed with available R code in the supplementary material, we choose $200$ bacterial DNA sequences from the SILVA~\cite{quast2012silva}
database used in DADA2~\cite{callahan2016dada2}. We align them and
take subsequences of length $450$. We  repeat the simulation
experiment described above using a reference population of $200$ sequences, pick samples of size $40$
 and compute the nearest-neighbor distances using the standard two-parameter Kimura distance~\cite[Section 10.4]{holmes2018modern}, we also compute the $160$ sample-to-population distances as before. This is repeated $500$ times, and we show the  two histograms in Figure \ref{fig:silva_distance_analysis}. 

Note that the two distributions, within sample and sample-to-population,  provide very similar first percentiles. The within sample approximation is $0.155$, and the sample-to-population value is
$ 0.146$. This tells us that we can take  $0.155$ as a threshold distance as a small distance
indicating "coincidence", this threshold
uncovers
 $5$ pairs of coincident sequences, which {\it were} in fact identical species.
For more details, see the reproducible code in the supplementary material.

\begin{figure}[htbp]
        \centering
\includegraphics[width=0.8\textwidth,height=0.4\textwidth]{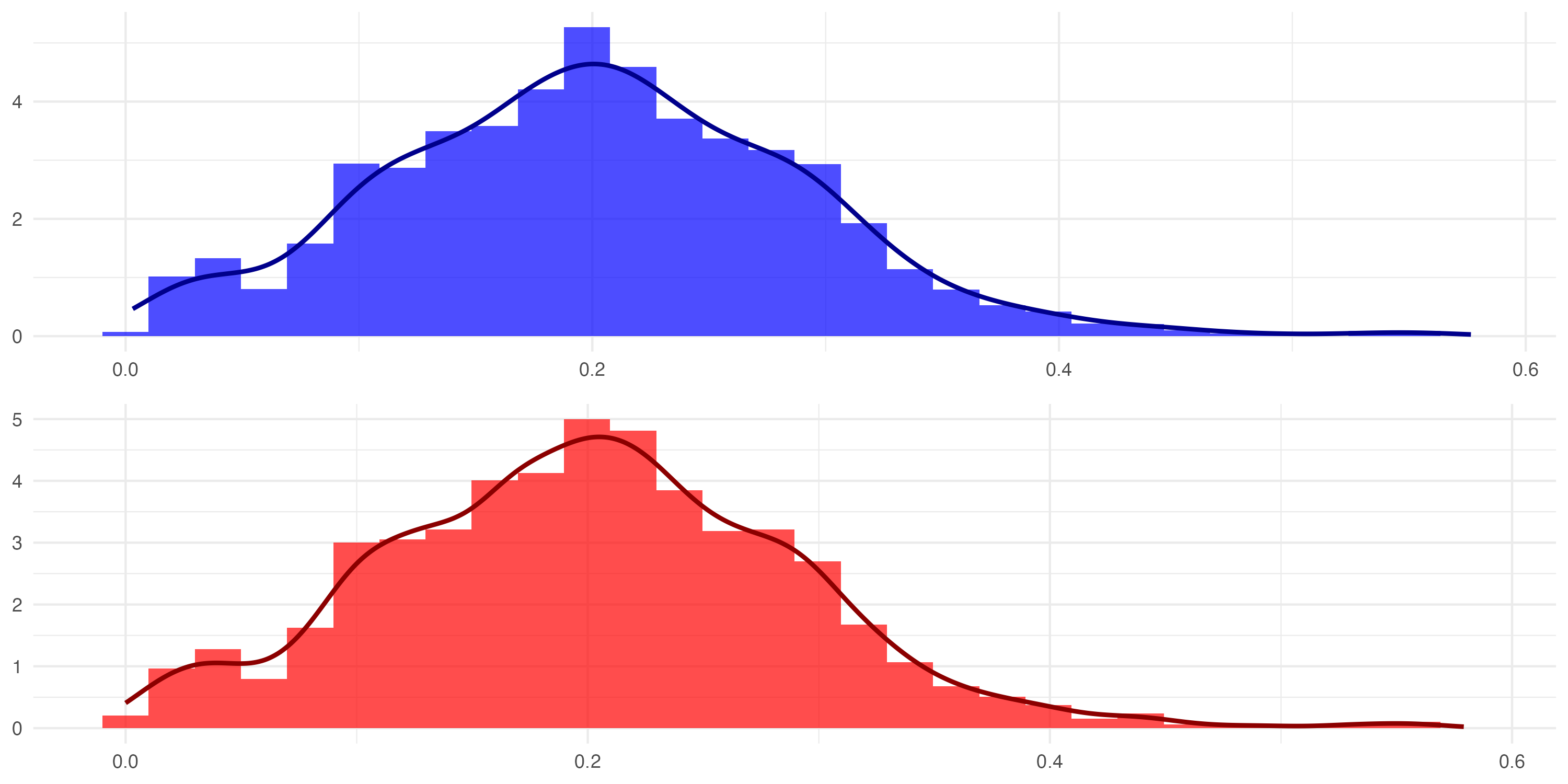}
        \caption{Distribution of nearest neighbor distances in the SILVA data. Top: Within-sample nearest-neighbor distances.  Bottom: Population-to-sample nearest neighbor distances.}
   \label{fig:silva_distance_analysis}
\end{figure}

\subsection{Prediction sets}\label{predictsec}
Suppose we have i.i.d.~data $(X_1,Y_1),\ldots,(X_n,Y_n)$, where $X_i$'s are explanatory variables taking value in some measurable space $(S,\mathcal{S})$ and $Y_i$'s are response variables taking value in some measurable space $(T, \mathcal{T})$. Suppose we are given a blackbox algorithm for computing prediction sets based on this data; that is, for each $n$, we have a map $P_n$ that produces a prediction set 
\[
P_n(X_{n+1}; (X_1,Y_1),\ldots, (X_n, Y_n))
\]
for $Y_{n+1}$ given $X_{n+1}$. We assume the measurability condition that the event $Y_{n+1}\in P_n(X_{n+1}; (X_1,Y_1),\ldots, (X_n, Y_n))$ is measurable. Also, we assume that the map $P_n$ is symmetric with respect to permutations of the data. Our goal is to estimate the coverage probability for this prediction set, that is, the conditional probability
\[
p_n := \P[Y_{n+1}\in P_n(X_{n+1}; (X_1,Y_1),\ldots, (X_n, Y_n)) | (X_1,Y_1),\ldots,(X_n, Y_n)].
\]
The leave-one-out estimate for this coverage probability is 
\begin{align*}
\hat{p}_n := \frac{1}{n}\sum_{i=1}^n 1_{\{Y_i \in P_{n-1}(X_i; (X_1,Y_1),\ldots, (X_{i-1}, Y_{i-1}), (X_{i+1}, Y_{i+1}), \ldots, (X_n, Y_n))\}}.
\end{align*}
The following theorem gives an upper bound on the error of this estimate. The example \ref{predictex} below specializes this to give more explicit
estimates.
\begin{thm}\label{predictthm}
Let $P_n$, $p_n$, and $\hat{p}_n$ be as above. Let $L_n$ be the list $(X_1,Y_1),\ldots,(X_n, Y_n)$. Define
\begin{align*}
&\delta_n' := \P(Y_{n+1}\in P_n(X_{n+1}; L_n)\Delta P_{n-1}(X_{n+1}; L_{n-1})),\\
&\delta_n'' := \P(Y_{n+1}\in P_{n-1}(X_{n+1}; L_{n-1})\Delta P_{n-2}(X_{n+1}; L_{n-2})).
\end{align*}
Assume that $n\ge 3$. Then
\[
\E[(\hat{p}_n - p_n)^2]\le 2\delta_n' + \frac{4(n-1)}{n}\delta_n'' + \frac{1}{2n}.
\]
\end{thm}
\begin{proof}
To put this problem in the framework of Theorem \ref{mainthm}, we replace the space $(S,\mathcal{S})$ be the space $(S\times T, \mathcal{S}\times \mathcal{T})$ in the present setting. Next, we define
\begin{align*}
A((x_1,y_1),\ldots,(x_n,y_n)) &:= \{(x,y)\in S\times T: y \in P_n(x; (x_1,y_1),\ldots,(x_n,y_n))\}.
\end{align*}
By the assumed measurability condition, $A$ is a measurable set-valued map. We define $A'$ and $A''$ similarly, replacing $n$ by $n-1$ and $n-2$, respectively. Then note that 
\begin{align*}
\mu(A(X_1,\ldots,X_n)) &= \P[(X_{n+1}, Y_{n+1}) \in A(X_1,\ldots,X_n)| (X_1,Y_1),\ldots,(X_n, Y_n)]\\
&= \P[Y_{n+1}\in P_n(X_{n+1}; (X_1,Y_1),\ldots,(X_n, Y_n))| (X_1,Y_1),\ldots,(X_n, Y_n)] \\
&= p_n,
\end{align*}
and 
\begin{align*}
\frac{1}{n}\sum_{i=1}^n 1_{\{X_i \in A'(X_1,\ldots,X_{i-1},X_{i+1},\ldots,X_n)\}} &= \hat{p}_n.
\end{align*}
Thus, we only have to compute the upper bounds on the quantities $\theta$, $\delta'$ and $\delta''$ from Theorem~\ref{mainthm}. But clearly, $\theta \le \frac{1}{4}$, $\delta' = \delta_n'$ and $\delta'' = \delta_n''$. This completes the proof.
\end{proof}

\medskip
\noindent\textit{Multi-step prediction.}
The one-step case treated above corresponds to predicting a single future response. A natural extension is to prediction sets for a block $(Y_{n+1},\ldots,Y_{n+m})$ given the corresponding covariates $(X_{n+1},\ldots,X_{n+m})$. One would then work with a set-valued map taking values in $T^m$, and the natural analogue of the present estimator would be a block-exclusion or leave-$m$-out procedure rather than simple leave-one-out. The same philosophy should continue to apply, with an error bound controlled by the stability of the prediction set under deletion of a block of observations; in particular, one should expect the constants to depend on $m$. Since the notation becomes considerably heavier and the cleanest formulation seems to require a separate statement, we do not pursue this extension here.

\begin{ex}[Linear regression prediction intervals]
\label{predictex}
As an application, consider the prediction intervals given by ordinary linear regression (OLS) without an intercept term. The reason for leaving out the intercept is just to make the analysis less cumbersome. The same analysis can be carried out with an intercept term present. Leaving out the intercept term is reasonable if the covariates and the response variable are centered. 

Here, we take the $X_i$'s to be $p$-dimensional and the $Y_i$'s to be real-valued. Note that we are not assuming that the true relation between $X_i$ and $Y_i$ is given by a linear regression model; we have only decided to use linear regression theory to construct prediction intervals. Thus, the true coverage probability may not be equal to the coverage probability that we are aiming for, and it is therefore important to estimate the true coverage probability. This gives relevance to Theorem \ref{predictthm}.

First, let us recall the form of the prediction interval from linear regression theory. The formula for the estimated parameter vector $\hat{\beta}$ in the absence of an intercept is given by
\[
\hat{\beta} = (X^TX)^{-1}X^TY,
\]
where $X$ is the $n\times p$ matrix whose $i^{\mathrm{th}}$ row is $X_i^T$ (thinking of $X_i$ as a column vector), and $X^T$ denotes the transpose of $X$, provided that $X$ has rank $p$. Similarly, $Y$ is the $n\times 1$ vector whose $i^{\mathrm{th}}$ component is $Y_i$. Given a newly drawn vector of covariates $X_{n+1}\in \R^p$, the predicted value of the corresponding $Y_{n+1}$ is $\hat{Y}_{n+1} := X_{n+1}^T \hat{\beta}$. The theoretical covariance matrix of $\hat{\beta}$ conditional on $X$, assuming that the linear regression model is true, is given by $\sigma^2 (X^TX)^{-1}$, where $\sigma^2$ is the variance of the errors in the regression model. Thus, the theoretical conditional variance of $\hat{Y}_{n+1}$ is given by $\sigma^2X_{n+1}^T (X^TX)^{-1}X_{n+1}$. Thus, given $X$ and $X_{n+1}$, $Y_{n+1} - \hat{Y}_{n+1}$ is theoretically approximately a normal random variable with mean zero and variance $\sigma^2(1+X_{n+1}^T(X^TX)^{-1}X_{n+1})$. The error variance $\sigma^2$ is estimated by 
\begin{align}\label{hatsigma}
\hat{\sigma}^2 := \frac{1}{n-p}\sum_{i=1}^n (Y_i - \hat{Y_i})^2 = \frac{1}{n-p} Y^T(I-X(X^TX)^{-1}X^T)Y,
\end{align}
where $\hat{Y}_i$ is the fitted value of $Y_i$. Thus, the level $1-\alpha$ prediction interval for $Y_{n+1}$ that is usually used in practice is the one 
\begin{align}\label{predint}
\biggl[\hat{Y}_{n+1} \pm z_{1-\frac{\alpha}{2}}\hat{\sigma}\sqrt{1+X_{n+1}^T(X^TX)^{-1}X_{n+1}}\biggr],
\end{align}
where $z_a$ denotes the $a^{\mathrm{th}}$ quantile of the standard normal distribution.

Suppose that we are given data $(X_1,Y_1),\ldots,(X_n,Y_n)$ that are i.i.d.~from some joint distribution on $\R^p \times \R$, which need not be the ordinary linear regression model, and we decide to use the prediction interval displayed in equation \ref{predint} to predict $y$ given a new $x$. Let $p_n$ be the true coverage probability of this prediction interval (given the data), and let $\hat{p}_n$ be the leave-one-out estimate of $p_n$. The following corollary of Theorem \ref{predictthm} shows that in this setting, $\hat{p}_n = p_n + O_P(n^{-1/2})$ as $n\to \infty$ (with all else remaining fixed), under some mild assumptions.
\begin{cor}\label{predictcor}
Suppose that $(X_1,Y_1),\ldots,(X_n,Y_n)$ are i.i.d.~from some joint distribution on $\R^p \times \R$. Assume that the conditional distribution of $Y_1$ given $X_1=x$ has a density with respect to Lebesgue measure that is bounded above by a finite constant that does not depend on $x$. Further, assume that $Y_1$ and the components of $X_1$ have sub-Gaussian tails, and that $(X_1,Y_1)$ has a bounded probability density with respect to Lebesgue measure on $\R^p\times \R$. Let $p_n$ be the coverage probability of the OLS prediction interval and $\hat{p}_n$ be its leave-one-out estimate. Then
\[
\E[(\hat{p}_n - p_n)^2] \le \frac{C}{n},
\]
where $C$ depends only on the dimension $p$, the level $\alpha$, and the joint law of $(X_1,Y_1)$.
\end{cor}
It is not hard to show that the assumption that $(X_1,Y_1)$ has a joint density implies that $X$ has rank $p$, and hence $\hat{\beta}$ is well-defined. Corollary \ref{predictcor} is proved in Appendix \ref{predictproof}.

\subsubsection{Simulations}
In Figure \ref{fig:mse_comparison}, we plot the mean squared error (MSE) between the leave-one-out (LOO) estimate of coverage probability and the true coverage probability against $1/n$ for two scenarios. On the left, we have data generated from a linear model with standard normal errors, where both the true coverage and its LOO estimate converge to the target coverage of $0.95$. On the right, we have data where the true relationship is nonlinear ($Y = X_1^2 + 0.5X_2 + \varepsilon$, $\varepsilon \sim N$) but a linear model is fitted, resulting in both estimates converging to a coverage lower than the target due to model mis-specification. The approximately linear relationship between MSE and $1/n$ in both cases confirm the theoretical result that $\mathbb{E}[(\hat{p}_n - p_n)^2] \leq C/n$. 

\begin{figure}[htbp]
    \centering
    \includegraphics[width=0.9\textwidth, height=0.5\textwidth]{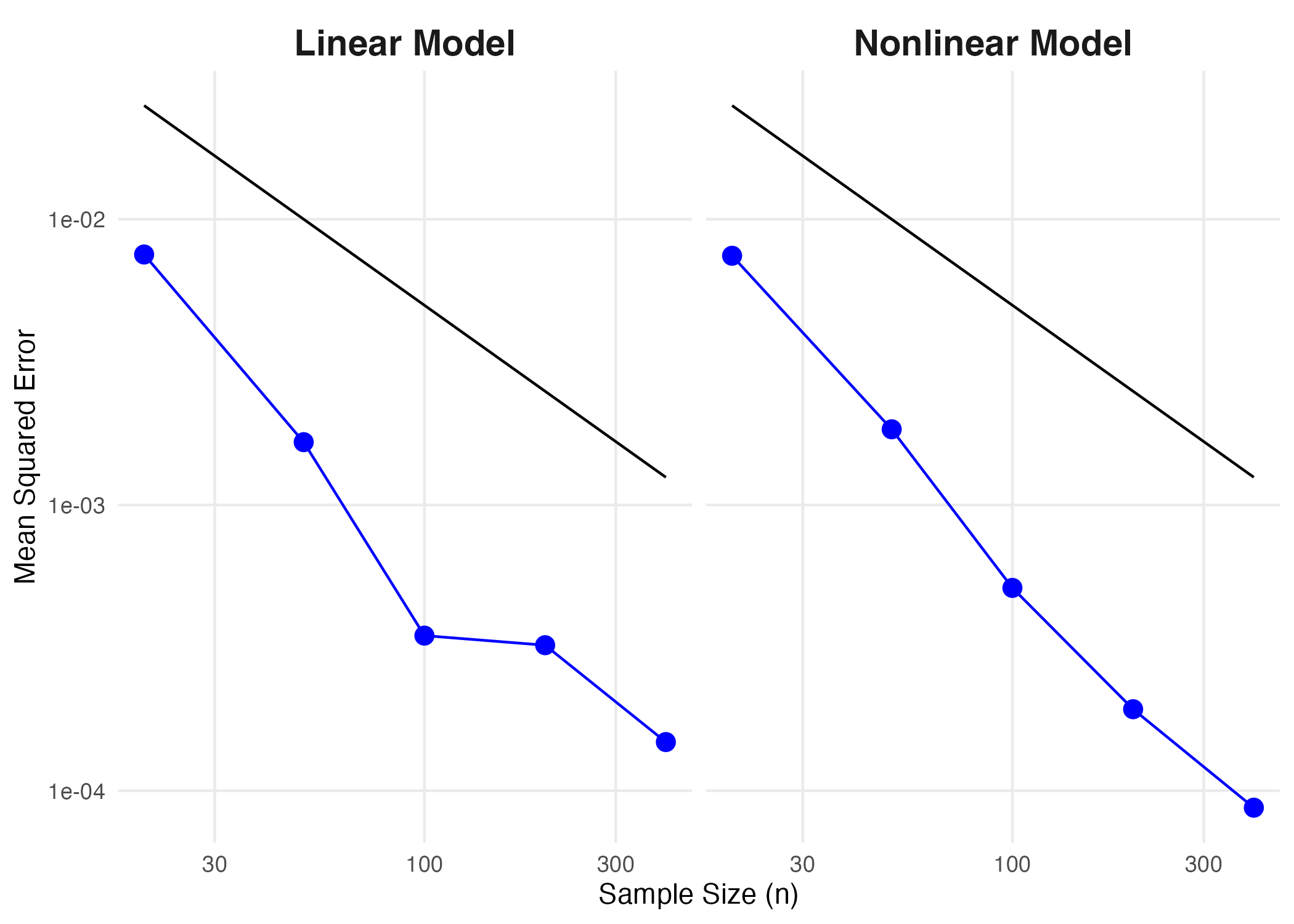}
    \caption{MSE of exclusion coverage probabilities between comparing the LOO estimate of coverage probability and the true coverage probability plotted against $1/n$ for two scenarios. \textbf{Left:} Data generated from a linear model where both the true coverage and its  estimate converge to the target coverage of $0.95$. \textbf{Right:} Data where the true relationship is nonlinear but a linear model is fitted, resulting in both estimates converging to a coverage lower than the target due to model mis-specification. The approximately linear relationship between MSE and $1/n$ in both cases confirm the theoretical result that $\mathbb{E}[(\hat{p}_n - p_n)^2] \leq C/n$.}
    \label{fig:mse_comparison}
\end{figure}

\subsubsection{Cholestyramine data}
As an application to real data, we consider an old data set from \citet{efron1991statistical}. The response variable is an improvement score for $164$ men taking cholestyramine to reduce cholesterol. The predictor variable is their compliance. The data are available in the bootstrap package in R, and the complete R code to reproduce our analysis is available in the supplementary materials.

We do a series of experiments using half the data to compute our leave-one-out coverage
estimates and using the other half of the data as the `truth' for estimating the coverage probability. This is repeated 
$1000$ times.

For both linear and quadratic regressions, the values are very similar. In the quadratic case, the mean true coverage was
$0.936$ and the mean estimated coverage (by leave-one-out) was $0.934$. It was even closer in the linear case
(true: $0.9336$ versus LOO: $0.9335$).


\subsubsection{Literature review}
The use of leave-one-out cross validation is pervasive in regression contexts, both for choosing tuning parameters and for 
evaluating models. Theoretical results, unfortunately, are few. 
Asymptotic results do exist  for the Lasso~\cite{austern2020asymptotics, yang2007consistency, homrighausen2014leave},  but as far as we know, ours are new finite sample bounds. It is a challenging problem to adapt Theorem \ref{predictthm} to more general predictors (we are working on it).

\subsection{Connection to a theorem of Devroye and Wagner}\label{devroyesec}
It was pointed out to us by Louigi Addario-Berry that our main result bears a resemblance with a theorem proved by Luc Devroye in his Ph.D.~thesis~\cite[Theorem 6.1]{devroye76}, based on a prior technical report with his advisor Terry Wagner (see also \cite{rogerswagner78}). An updated version of this theorem appears in the book by \citet*[Theorem 24.2]{devroyeetal13}. We now discuss this result and its connection to our Theorem \ref{mainthm}. 

The setting of the Devroye--Wagner theorem is as follows. The data consists of an i.i.d.~sequence $D_n$ of pairs $(X_1,Y_1),\ldots,(X_n, Y_n)$, where the $X_i$'s take value in some arbitrary space and the $Y_i$'s are $0$-$1$ valued. The goal is to construct a {\it classifier} $g_n$ based on this data, which, given a new data point $X_{n+1}$, will output $g_n(X_{n+1},D_n)\in \{0,1\}$ as the predicted value of $Y_{n+1}$. We assume that the classifier is {\it symmetric} meaning that $g_n(x,D_n') = g_n(x,D_n)$, where $D_n'$ is any permutation of $D_n$.

Define $L_n$ to be the misclassification rate of $g_n$ conditional on the data $D_n$; that is,
\[
L_n = \P(g_n(X_{n+1}, D_n)\ne Y_{n+1} | D_n).
\]
Suppose we have another symmetric classifier $g_{n-1}$ for samples of size $n-1$, and suppose we estimate $L_n$ using a leave-one-out procedure based on $g_n$:
\[
\hat{L}_n := \frac{1}{n}\sum_{i=1}^n 1_{\{g_{n-1}(X_i, D_{n,i}) \ne Y_i\}},
\]
where $D_{n,i}$ is obtained from $D_n$ by omitting the pair $(X_i,Y_i)$. The Devroye--Wagner theorem~\cite[Theorem 24.2]{devroyeetal13} says that 
\begin{align}\label{devroye}
\E[(\hat{L}_n - L_n)^2] \le \frac{1}{n} + 6\, \P(g_n(X_{n+1}, D_n) \ne g_{n-1}(X_{n+1}, D_{n-1})).
\end{align}
A very similar bound can be obtained from our Theorem \ref{mainthm}, as follows. In addition to $g_{n-1}$, we need a symmetric prediction rule $g_{n-2}$ for samples of size $n-2$. Let $A(D_n) := \{(x,y): y = g_n(x, D_n)\}$, and define $A'$ and $A''$ similarly. Then 
\begin{align*}
\delta' &= \E[\mu(A(D_n) \Delta A'(D_{n-1}))] \\
&= \E[\P(\{Y_{n+1}= g_n(X_{n+1}, D_n)\}\Delta\{Y_{n+1}= g_{n-1}(X_{n+1}, D_{n-1})\} | D_n, X_{n+1})]\\
&= \P(\{Y_{n+1}= g_n(X_{n+1}, D_n)\}\Delta\{Y_{n+1}= g_{n-1}(X_{n+1}, D_{n-1})\})\\
&\le \P(g_n(X_{n+1}, D_n) \ne g_{n-1}(X_{n+1}, D_{n-1})).
\end{align*}
Similarly,
\[
\delta'' \le \P(g_{n-1}(X_{n+1}, D_{n-1}) \ne g_{n-2}(X_{n+1}, D_{n-2})).
\]
Bounding $\theta$ by $\frac{1}{4}$, we get the bound
\begin{align*}
\E[(\hat{L}_n - L_n)^2] &\le \frac{1}{4n} + 4\, \P(g_n(X_{n+1}, D_n) \ne g_{n-1}(X_{n+1}, D_{n-1}))\\
&\qquad + 4\, \P(g_{n-1}(X_{n+1}, D_{n-1}) \ne g_{n-2}(X_{n+1}, D_{n-2})).
\end{align*}
In practice, this is essentially of the same order as the Devroye--Wagner bound \eqref{devroye}. As noted after Theorem \ref{mainthm}, a weaker formulation using only $A'$ is possible. The role of $A''$ in our statement is to encode the second-order dependence between distinct leave-one-out summands, which is exactly what the proof in Appendix \ref{mainproof} controls. We keep the $A''$-formulation because it separates the approximation term $\delta'$ from the dependence term $\delta''$, and in our applications the latter is canonical and usually leads to a sharper bound than reducing everything to $\delta'$ alone.

\subsection{Connection to algorithmic stability}
Roughly speaking, algorithmic stability is the notion that the output of an algorithm, whose input is a set of i.i.d.~observations, should not change very much under omitting any one of the observations. It is clear that this idea is pertinent for applications of Theorem \ref{mainthm}. Indeed, Theorem \ref{mainthm} says that under a certain kind of algorithmic stability, the leave-one-out estimate of the size of a random set is a good estimate. The various applications we have worked out in this paper are all about showing that the respective algorithms are stable. The main difference with the prior literature is that the usual versions of algorithmic stability mainly look at point estimates, whereas we are looking at random sets. We refer to Chapter 6 in the forthcoming monograph of \citet*{angelopoulosetal24} for references to the algorithmic stability literature. 



\subsection*{Acknowledgements}
We thank Louigi Addario-Berry, David Aldous, Eugenio Regazzini, Richard Stanley, Tim Sudijono, and Sandy Zabell 
for their help.
PD was supported by NSF grant DMS-1954042.
SC was supported in part by NSF grants DMS-2113242
and DMS-2153654.


\appendix


\section{Appendix}
\subsection{Proof of Theorem \ref{mainthm}}\label{mainproof}
The proof needs two lemmas. For $1\le i\le n$, define 
\[
K_i := 1_{\{X_i \in A'(X_1,\ldots,X_{i-1},X_{i+1},\ldots,X_n)\}},
\]
and let $I := \sum_{i=1}^n K_i$. 
Next, define
\[
L_i := \E(K_i|X_1,\ldots,X_{i-1}, X_{i+1},\ldots,X_n),
\]
and let $I' := \sum_{i=1}^n L_i$.
\begin{lmm}\label{nnlmm}
We have
\[
\E[(I-I')^2] \le n\theta  + 2n(n-1)\delta''. 
\]
\end{lmm}
\begin{proof}
Expanding the square, we have
\begin{align*}
\E[(I-I')^2] &= \sum_{i=1}^n \E[(K_i-L_i)^2] + 2\sum_{1\le i<j\le n} \E[(K_i-L_i)(K_j-L_j)].
\end{align*}
First, note that 
\begin{align*}
\E[(K_i-L_i)^2] &= \E[\var(K_i|X_1,\ldots,X_{i-1},X_{i+1},\ldots,X_n)]\\
&= \E[L_i(1-L_i)] = \theta.
\end{align*}
Take any $1\le i<j\le n$. Let 
\[
K_i' := 1_{\{X_i\in A''(X_1,\ldots, X_{i-1}, X_{i+1},\ldots,X_{j-1},X_{j+1},\ldots,X_n)\}},
\]
and define
\[
L_i' := \E(K_i'|X_1,\ldots, X_{i-1}, X_{i+1},\ldots,X_{j-1},X_{j+1},\ldots,X_n). 
\]
Since 
\[
|K_i - K_i'| = 1_{\{X_i \in A'(X_1,\ldots,X_{i-1},X_{i+1},\ldots,X_n)\Delta A''(X_1,\ldots, X_{i-1}, X_{i+1},\ldots,X_{j-1},X_{j+1},\ldots,X_n) \}}, 
\]
it follows from our definition of measurable symmetric set-valued maps  that 
\begin{align*}
\E|K_i - K_i'| &= \delta''. 
\end{align*}
Now note that $K_i'$ has no dependence on $X_j$. This implies that $L_i'$ also has no dependence on $X_j$, and that 
\[
L_i' = \E(K_i'|X_1,\ldots, X_{i-1}, X_{i+1},\ldots,X_n).
\]
Thus, we get
\begin{align*}
\E|L_i - L_i'| &= \E|\E(K_i - K_i'|X_1,\ldots, X_{i-1}, X_{i+1},\ldots,X_n)|\le \E|K_i-K_i'|. 
\end{align*}
Combining the above observations, we get
\begin{align*}
|\E[(K_i-L_i)(K_j-L_j)] - \E[(K_i'-L_i')(K_j-L_j)]| &\le \E|(K_i-L_i)-(K_i'-L_i')|\\
&\le \E|K_i - K_i'| + \E|L_i - L_i'|\\
&\le2 \delta''.
\end{align*}
Now recall that $K_i'$ and $L_i'$ have no dependence on $X_j$, and $L_j$ is the conditional expectation of $K_j$ given $(X_l)_{l\ne j}$. Thus, 
\begin{align*}
\E[(K_i'-L_i')(K_j-L_j)] &= \E[(K_i'-L_i')\E(K_j-L_j|X_1,\ldots,X_{j-1},X_{j+1},\ldots,X_n)] = 0.
\end{align*}
Thus, we arrive at the conclusion that for each $1\le i<j\le n$,
\[
|\E[(K_i-L_i)(K_j-L_j)]|\le2 \delta''.
\]
This completes the proof.
\end{proof}
\begin{lmm}\label{nalmm}
We have
\[
\E\biggl|\frac{I'}{n} - \mu(A(X_1,\ldots,X_n))\biggr| \le \delta'.
\]
\end{lmm}
\begin{proof}
Note that by definition of $L_i$ and independence of the $X_j$'s,
\begin{align*}
L_i &= \mu(A'(X_1,\ldots,X_{i-1}, X_{i+1},\ldots,X_n)).
\end{align*}
Thus,
\begin{align*}
\frac{I'}{n} &= \frac{1}{n}\sum_{i=1}^n \mu(A'(X_1,\ldots,X_{i-1}, X_{i+1},\ldots,X_n)).
\end{align*}
By the inequality $|\mu(B)-\mu(C)|\le \mu(B\Delta C)$, this gives
\begin{align*}
\biggl|\frac{I'}{n} - \mu(A(X_1,\ldots,X_n))\biggr| &\le \frac{1}{n}\sum_{i=1}^n \mu(A'(X_1,\ldots,X_{i-1}, X_{i+1},\ldots,X_n)\Delta A(X_1,\ldots,X_n)).
\end{align*}
Taking expectation on both sides and using symmetry, we get the desired result.
\end{proof}
We are now ready to prove Theorem \ref{mainthm}.
\begin{proof}[Proof of Theorem \ref{mainthm}]
Combining Lemma \ref{nnlmm} and Lemma \ref{nalmm}, and observing that the random variables $I'/n$ and $\mu(A(X_1,\ldots,X_n))$ take value in $[0,1]$, we get
\begin{align*}
\E\biggl[\biggl(\mu(A(X_1,\ldots,X_n) - \frac{I}{n}\biggr)^2\biggr]&\le 2\E\biggl[\biggl(\frac{I}{n}-\frac{I'}{n}\biggr)^2\biggr] +  2\E\biggl[\biggl(\mu(A(X_1,\ldots,X_n)) - \frac{I'}{n}\biggr)^2\biggr]\\
&\le 2\E\biggl[\biggl(\frac{I}{n}-\frac{I'}{n}\biggr)^2\biggr] +  2\E\biggl|\mu(A(X_1,\ldots,X_n)) - \frac{I'}{n}\biggr|\\
&\le \frac{2\theta}{n} + \frac{4(n-1)\delta''}{n} + 2\delta'. 
\end{align*}
This completes the proof.
\end{proof}

\subsection{Proof of Theorem \ref{mainthm2}}\label{mainproof2}
The proof is minor modification of the proof of Theorem \ref{mainthm}. Let all notations be as in that proof. Fix $1\le i<j\le n$, and let $X_j^*$ be an independent copy of $X_j$. Define
\[
K_i^*:=
1_{\{X_i\in A'(X_1,\ldots,X_{i-1},X_{i+1},\ldots,X_{j-1},X_j^*,X_{j+1},\ldots,X_n)\}}.
\]
Also define
\[
L_i^*:=\E(K_i^*| X_1,\ldots,X_{i-1},X_{i+1},\ldots,X_{j-1},X_j^*,X_{j+1},\ldots,X_n).
\]
Since $K_i^*-L_i^*$ is measurable with respect to $\sigma(X_1,\ldots,X_{j-1},X_j^*,X_{j+1},\ldots,X_n)$, while
\[
\E(K_j-L_j| X_1,\ldots,X_{j-1},X_j^*,X_{j+1},\ldots,X_n)=0,
\]
we get
\[
\E[(K_i^*-L_i^*)(K_j-L_j)]=0.
\]
Therefore,
\begin{align*}
|\E[(K_i-L_i)(K_j-L_j)]| &= |\E[((K_i-L_i)-(K_i^*-L_i^*))(K_j-L_j)]|\\
&\le \E|(K_i-L_i)-(K_i^*-L_i^*)|\\ 
&\le \E|K_i-K_i^*|+\E|L_i-L_i^*|.
\end{align*}
Now, by conditioning on $(X_1,\ldots,X_{i-1},X_{i+1},\ldots,X_{j-1},X_j,X_j^*,X_{j+1},\ldots,X_n)$ and applying Jensen's inequality, $\E|L_i-L_i^*| \le \E|K_i-K_i^*|$.
By symmetry,
\begin{align*}
\E|K_i-K_i^*|&= \E[\mu(A'(X_1,\ldots,X_{i-1}, X_{i+1},\ldots,X_n)\Delta A'(X_1,\ldots,X_{i-1},X_{i+1},\\ 
&\qquad \qquad \ldots,X_{j-1},X_j^*,X_{j+1},\ldots,X_n))]\\
&= \E[\mu(A'(X_1,\ldots,X_{n-1})\Delta A'(X_1,\ldots,X_{n-2},X_n^*))]\\
&= \E[\mu(A'(X_1,\ldots,X_{n-1})\Delta A'(X_1,\ldots,X_{n-2},X_n))]= \rho.
\end{align*}
Hence,
\[
|\E[(K_i-L_i)(K_j-L_j)]|\le 2\rho.
\]
Summing over $1\le i<j\le n$ gives
\[
\E[(I-I')^2]\le n\theta+2n(n-1)\rho.
\]
Combining this with Lemma \ref{nalmm} gives the first bound in the theorem statement. Finally, by the triangle inequality,
\begin{align*}
&\mu(A'(X_1,\ldots,X_{n-1})\Delta A'(X_1,\ldots,X_{n-2},X_n))\\
&\le
\mu(A'(X_1,\ldots,X_{n-1})\Delta A(X_1,\ldots,X_n))\\
&\quad+
\mu(A(X_1,\ldots,X_n)\Delta A'(X_1,\ldots,X_{n-2},X_n)).
\end{align*}
Taking expectations and using symmetry shows that $\rho\le 2\delta'$, which yields the second bound in the theorem statement.

\subsection{Proof of Corollary \ref{predictcor}}\label{predictproof}
To prove Corollary \ref{predictcor}, we need several lemmas. In the following, $\|M\|$ denotes the operator norm of a matrix $M$. Throughout, we will work in the setting of Corollary \ref{predictcor}. Also, throughout, $C,C_1,C_2,\ldots$ will denote finite positive constants that may depend only on $p$ and the law of $(X_1,Y_1)$, whose values may change from line to line.
\begin{lmm}\label{predictlmm1}
For any $t\ge 0$,
\[
\P(\|X\|\ge t) \le C_1^n e^{-C_2 t^2}.
\]
\end{lmm}
\begin{proof}
Take any $\epsilon\in (0,1)$. Let $\mathbb{S}^{p-1}$ denote the Euclidean unit sphere in $\R^p$. It is a standard fact that there is a subset $A(\epsilon)\subseteq \mathbb{S}^{p-1}$ of size at most $C(p)\epsilon^{-(p-1)}$ (where $C(p)$ is a constant depending only on $p$) such that any point $x\in \mathbb{S}^{p-1}$ is within distance $\epsilon$ from some point $y\in A(\epsilon)$. Thus, 
\begin{align*}
\|Xx\| &\le \|Xy\| + \|X(x-y)\| \\
&\le \max_{z\in A(\epsilon)}\|Xz\| + \|X\|\|x-y\| \le \max_{z\in A(\epsilon)}\|Xz\| + \epsilon \|X\|.
\end{align*}
Choosing $\epsilon = \frac{1}{2}$ and maximizing the left side over $x\in \mathbb{S}^{p-1}$, we get 
\begin{align}\label{xmax}
\|X\| &\le 2 \max_{z\in A(\frac{1}{2})}\|Xz\|.
\end{align}
Take any $z\in A(\frac{1}{2})$. For any $\alpha >0$,
\begin{align*}
\E(e^{\alpha \|Xz\|^2}) &= \E\biggl[\exp\biggl(\alpha \sum_{i=1}^n (X_i^Tz)^2 \biggr)\biggr]\\
&= \prod_{i=1}^n \E(e^{\alpha (X_i^Tz)^2}) = [\E(e^{\alpha (X_1^Tz)^2})]^n.
\end{align*}
By the sub-Gaussian tail assumption, this shows that if $\alpha$ is chosen small enough, then $\E(e^{\alpha \|Xz\|^2})\le C^n$. By Markov's inequality, this gives
\begin{align*}
\P(\|Xz\|\ge t) &= \P(e^{\alpha \|Xz\|^2} \ge e^{\alpha t^2}) \le e^{-\alpha t^2} C^n.
\end{align*}
By the inequality \eqref{xmax} and a union bound, this gives
\begin{align*}
\P(\|X\|\ge t) &\le \P\biggl(2 \max_{z\in A(\frac{1}{2})}\|Xz\| \ge t\biggr)\\
&\le \sum_{z\in A(\frac{1}{2})}\P(\|Xz\|\ge \textstyle{\frac{1}{2}}t)\le C_1^n e^{-C_2 t^2}.
\end{align*}
This completes the proof.
\end{proof}
\begin{lmm}\label{predictlmm2}
For any $\alpha >1$ and $z\in \mathbb{S}^{p-1}$,
\[
\E(e^{-\alpha(X_1^Tz)^2}) \le C_1 e^{-C_2 \log \alpha},
\]
and the same bound also holds for $\E(e^{-\alpha(Y_1 - X_1^Tb)^2})$ for any $\alpha > 1$ and $b\in \R^p$.
\end{lmm}
\begin{proof}
First, note that 
\begin{align*}
\E(e^{-\alpha(X_1^Tz)^2}) &= \E(e^{-\alpha(X_1^Tz)^2}; |X_1^T z| \le \alpha^{-\frac{1}{2}}\log \alpha) + \E(e^{-\alpha(X_1^Tz)^2}; |X_1^T z| > \alpha^{-\frac{1}{2}}\log \alpha)\\
&\le \P(|X_1^T z| \le \alpha^{-\frac{1}{2}}\log \alpha) + e^{-(\log \alpha)^2}.
\end{align*}
Next, note that
\begin{align*}
\P(|X_1^T z| \le \alpha^{-\frac{1}{2}}\log \alpha) &\le \P(|X_1^T z| \le \alpha^{-\frac{1}{2}}\log \alpha, \, \|X_1\|\le \log \alpha, \, |Y_1|\le \log \alpha) \\
&\qquad + \P(\|X_1\|> \log \alpha) + \P(|Y_1|>\log \alpha).
\end{align*}
By the sub-Gaussian tail assumption, 
\[
\P(\|X_1\|> \log \alpha) + \P(|Y_1|> \log \alpha) \le C_1 e^{-C_2(\log\alpha)^2}.
\]
By the assumption that $(X_1, Y_1)$ has a bounded probability density and the fact that $z$ is a unit vector,
\begin{align*}
&\P(|X_1^T z| \le \alpha^{-\frac{1}{2}}\log \alpha, \, \|X_1\|\le \log \alpha,\, |Y_1|\le \log \alpha) \\
&\le C_1 \vol(\{(x,y)\in \R^p\times \R: |x^T z| \le \alpha^{-\frac{1}{2}}\log \alpha,\,  \|x\|\le \log \alpha,\, |y|\le \log \alpha\})\\
&\le C_2\alpha^{-\frac{1}{2}}(\log \alpha)^{p+1}.
\end{align*}
Combining the above inequalities, we get 
\begin{align*}
\E(e^{-\alpha(X_1^Tz)^2}) &\le C_1\alpha^{-\frac{1}{2}}(\log \alpha)^{p+1} + C_2 e^{-C_3 (\log \alpha)^2} + e^{-(\log \alpha)^2} \le C_4e^{-C_5\log \alpha}.
\end{align*}
This completes the proof of the first inequality. The second inequality follows similarly, by replacing $X_1^Tz$ with $Y_1 - X_1^Tb$ in every step above. Note that we do not need $b$ to be a unit vector for this bound, because the volume estimate does not need it, unlike in the first case.
\end{proof}
\begin{lmm}\label{predictlmm3}
For any $k\ge 1$, there are positive constants $C_1(k)$ and $C_2(k)$ depending only on $k$, $p$ and the law of $X_1$, such that if $n\ge C_1(k)$, then 
\[
\E[\|(X^TX)^{-1}\|^k] \le C_2(k)n^{-k}.
\]
\end{lmm}
\begin{proof}
First, note that $\|(X^TX)^{-1}\|$ is inverse of the smallest eigenvalue of $X^TX$. Thus,
\begin{align}\label{xtx}
\|(X^TX)^{-1}\|^{-1} = \min_{x\in \mathbb{S}^{p-1}} x^TX^TXx = \min_{x\in \mathbb{S}^{p-1}} \|Xx\|^2. 
\end{align}
Let $A(\epsilon)$ be as in the proof of Lemma \ref{predictlmm1}. Take any $x\in \mathbb{S}^{p-1}$ and $y\in A(\epsilon)$ such that $\|x-y\|\le \epsilon$. Then 
\begin{align*}
\|Xx\| &\ge \|Xy\| - \|X(x-y)\|\\
&\ge \min_{z\in A(\epsilon)} \|Xz\| - \epsilon\|X\|.
\end{align*}
Minimizing the left side over $x\in \mathbb{S}^{p-1}$, and applying the identity \eqref{xtx}, we get
\[
\|(X^TX)^{-1}\|^{-\frac{1}{2}} \ge \min_{z\in A(\epsilon)} \|Xz\| - \epsilon\|X\|.
\]
Thus, for any $t> 1$,
\begin{align*}
\P(\|n(X^TX)^{-1}\|\ge t) &= \P(\|(X^TX)^{-1}\|^{-\frac{1}{2}}\le \sqrt{n}t^{-\frac{1}{2}})\\
&\le \P\biggl(\min_{z\in A(\epsilon)} \|Xz\|\le 2\sqrt{n}t^{-\frac{1}{2}}\biggr) + \P(\epsilon \|X\|\ge \sqrt{n}t^{-\frac{1}{2}})\\
&\le \sum_{z\in A(\epsilon)}\P(\|Xz\|\le 2\sqrt{n}t^{-\frac{1}{2}}) + \P(\|X\|\ge \epsilon^{-1}\sqrt{n}t^{-\frac{1}{2}}).
\end{align*}
Take any $\alpha>1$ and $z\in A(\epsilon)$. Then 
\begin{align*}
\P(\|Xz\| \le 2\sqrt{n}t^{-\frac{1}{2}}) &= \P(e^{-\alpha \|Xz\|^2}\ge e^{-4\alpha n t^{-1}})\\
&\le e^{4\alpha nt^{-1}}\E(e^{-\alpha \|Xz\|^2}) = e^{4\alpha nt^{-1}}[\E(e^{-\alpha(X_1^Tz)^2})]^n.
\end{align*}
Combining the above inequalities and invoking Lemma \ref{predictlmm1} and Lemma \ref{predictlmm2}, we get
\begin{align*}
\P(\|n(X^TX)^{-1}\|\ge t) &\le C_1^n e^{C_2 \log \epsilon + 4\alpha nt^{-1} - C_3 n\log \alpha } + C_4^n e^{-C_5\epsilon^{-2}n t^{-1}}.
\end{align*}
Note that here $t> 1$ is given, and $\epsilon\in (0,1)$ and $\alpha >1$ are arbitrary. Let us now choose $\epsilon = t^{-1}$ and $\alpha = t$. Then the above bound gives
\begin{align}\label{xxineq}
\P(\|n(X^TX)^{-1}\|\ge t) &\le C_1 e^{-C_2 n\log t}.
\end{align}
Take any $k\ge 1$. Then by the above inequality,
\begin{align*}
\E[\|n(X^TX)^{-1}\|^k] &= \int_0^\infty k t^{k-1}\P(\|n(X^TX)^{-1}\|\ge t) dt\\
&\le 1 + \int_1^\infty k t^{k-1}\P(\|n(X^TX)^{-1}\|\ge t) dt\\
&\le 1 + C_1 k \int_1^\infty  t^{k-1-C_2 n} dt.
\end{align*}
If $n> (k-1)/C_2$, the right side is bounded by a finite constant that depends only on $k$. This completes the proof.
\end{proof}

\begin{lmm}\label{predictlmm4}
For any $t\ge 1$,
\[
\P(\|\hat{\beta}\|\ge t) \le C_1 e^{-C_2 n\log t} + C_3^n e^{-C_4 n \sqrt{t}}. 
\]
\end{lmm}
\begin{proof}
By inequality \eqref{xxineq},
\begin{align*}
\P(\|\hat{\beta}\|\ge t) &= \P(\|(X^TX)^{-1} X^TY\|\ge t)\\
&\le \P(\|n (X^TX)^{-1}\|\ge \sqrt{t}) + \P(\|n^{-1}X^TY\|\ge \sqrt{t})\\
&\le C_1 e^{-C_2 n\log t} + \P(\|X\|\ge \sqrt{n}t^{\frac{1}{4}}) + \P(\|Y\|\ge \sqrt{n}t^{1/4}).
\end{align*}
By Lemma \ref{predictlmm1},
\begin{align*}
\P(\|X\|\ge \sqrt{n}t^{\frac{1}{4}}) &\le C_1^n e^{-C_2 n\sqrt{t}}.
\end{align*}
By the sub-Gaussian tail assumption, with a small enough choice of $\alpha$, we have
\begin{align}\label{ytail}
\P(\|Y\|\ge \sqrt{n}t^{\frac{1}{4}}) &\le e^{-\alpha n \sqrt{t}}\E(e^{\alpha \|Y\|^2})\notag \\
&\le e^{-\alpha n \sqrt{t}}[\E(e^{\alpha Y_1^2})]^n\notag \\
&\le C_1^n e^{-\alpha n \sqrt{t}}. 
\end{align}
This completes the proof.
\end{proof}
\begin{lmm}\label{predictlmm5}
For any $k\ge 1$, there are positive constants $C_1(k)$ and $C_2(k)$ depending only on $k$, $p$ and the law of $X_1$, such that if $n\ge C_1(k)$, then  $\E(\hat{\sigma}^{-k})$ and $\E(\hat{\sigma}^k)$ are both bounded by $C_2(k)$. 
\end{lmm}
\begin{proof}
Take some $t>1$ and $\epsilon \in (0,1)$. Let $B_\epsilon$ be a collection of points in the ball $B(0,t)$ of radius $t$ centered at the origin in $\R^p$, such that the union of the balls of radius $\epsilon$ around the points in $B_\epsilon$ contains $B(0,t)$. By a standard argument, one can show that $B_\epsilon$ can be chosen such that $|B_\epsilon|\le Ct^p \epsilon^{-p}$. Take any $a\in B(0,t)$, and some $b\in B_\epsilon$ such that $\|a-b\|\le \epsilon$. Then
\begin{align*}
\|Y - Xa\| &\ge \|Y-Xb\| - \epsilon \|X\|.
\end{align*}
This shows that 
\begin{align*}
\min_{a\in B(0,t)} \|Y-Xa\|\ge \min_{b\in B_\epsilon} \|Y-Xb\|- \epsilon \|X\|.
\end{align*}
Thus, for any $s\in (0,\frac{1}{2})$,
\begin{align*}
\P\biggl(\min_{a\in B(0,t)} \|Y-Xa\|\le \sqrt{n}s\biggr) &\le \P\biggl(\min_{b\in B_\epsilon} \|Y-Xb\|\le 2\sqrt{n}s\biggr) + \P(\epsilon\|X\|\ge \sqrt{n}s) \\
&\le \sum_{b\in B_\epsilon} \P(\|Y-Xb\|\le 2\sqrt{n} s) + \P(\|X\|\ge \epsilon^{-1}\sqrt{n}s).
\end{align*}
Take any $b\in B_\epsilon$ and any $\alpha > 1$. By Lemma \ref{predictlmm2}, 
\begin{align*}
\E(e^{-\alpha \|Y-Xb\|^2}) = [\E(e^{-\alpha(Y_1-X_1^Tb)^2})]^n\le C_1^n e^{-C_2 n \log \alpha}.
\end{align*}
Thus, 
\begin{align*}
\P(\|Y-Xb\|\le 2\sqrt{n} s) &\le e^{4\alpha n s^2}\E(e^{-\alpha \|Y-Xb\|^2}) \le   C_1^n e^{4\alpha n s^2-C_2 n \log \alpha}.
\end{align*}
By Lemma \ref{predictlmm1}, 
\[
\P(\|X\|\ge \epsilon^{-1}\sqrt{n}s)\le C_1^n e^{-C_2\epsilon^{-2}n s^2}.
\]
Combining the above, we get
\begin{align*}
\P\biggl(\min_{a\in B(0,t)} \|Y-Xa\|\le \sqrt{n}s\biggr) &\le C_1^n t^p \epsilon^{-p}e^{4\alpha n s^2-C_2 n \log \alpha} + C_1^n e^{-C_2\epsilon^{-2}n s^2}.
\end{align*}
Choosing $\epsilon = s^2$ and $\alpha = s^{-2}$, this gives
\begin{align}\label{tpineq}
\P\biggl(\min_{a\in B(0,t)} \|Y-Xa\|\le \sqrt{n}s\biggr) &\le t^p e^{C n \log s},
\end{align}
provided that $n\ge C_3$. Now recall that 
\[
\hat{\sigma}^2 = \frac{1}{n-p}\min_{b\in \R^p}\|Y-Xb\|^2,
\]
and that the minimum on then right is attained at $b=\hat{\beta}$. Thus, for any $t\ge 2$,
\begin{align*}
\P(\hat{\sigma}^{-1} \ge t)  &= \P(\hat{\sigma}^2 \le t^{-2}, \, \|\hat{\beta}\|\le t) + \P(\|\hat{\beta}\|> t)\\
&\le \P\biggl(\min_{b\in B(0,t)}\|Y-Xb\| \le \sqrt{n} t^{-1}\biggr) + \P(\|\hat{\beta}\|> t).
\end{align*}
Thus, by Lemma \ref{predictlmm4} and inequality \eqref{tpineq}, we get
\begin{align*}
\P(\hat{\sigma}^{-1} \ge t) &\le t^p e^{-C_1n \log t} + C_2 e^{-C_3 n\log t} + C_4^n e^{-C_5 n \sqrt{t}}. 
\end{align*}
It is easy to see that this gives the desired upper bound on $\E(\hat{\sigma}^{-k})$.

Next, by the formula displayed in equation \eqref{hatsigma}, we get
\begin{align*}
\hat{\sigma}^2 \le \frac{1}{n-p} (\|Y\|^2 + \|(X^TX)^{-1}\| \|X\|\|Y\|).
\end{align*}
By the tail bound from equation \eqref{ytail}, we get that for any $k$,
\begin{align}\label{ymoment}
\E(\|Y\|^k) &\le C(k) n^{\frac{1}{2}k}.
\end{align}
Thus, we have
\begin{align*}
\E(\hat{\sigma}^k) &\le C(k) n^{-\frac{1}{2}k}[\E(\|Y\|^k) + \E(\|(X^TX)^{-1}\|^k \|X\|^k\|Y\|^k)]\\
&\le C(k) + C(k) [\E(\|(X^TX)^{-1}\|^{3k}) \E(\|X\|^{3k}) \E(\|Y\|^{3k})]^{\frac{1}{3}}.
\end{align*}
Applying Lemma \ref{predictlmm1}, Lemma \ref{predictlmm3}, and inequality \eqref{ymoment} to get upper bounds for the three expectations on the right, we get the desired bound.
\end{proof}

We are now ready to prove Corollary \ref{predictcor}.
\begin{proof}[Proof of Corollary \ref{predictcor}]
Let $\delta_n'$ and $\delta_n''$ be as in Theorem \ref{predictthm}. 
By the assumption about the conditional density of $Y_1$ given $X_1=x$, we have that
\begin{align*}
\delta_n' &\le C\E[ \lambda(P_n(X_{n+1}; L_n)\Delta P_{n-1}(X_{n+1}; L_{n-1}))],
\end{align*}
where $\lambda$ is Lebesgue measure on $\R$. A similar bound holds for $\delta_n''$. 

Let $\tX$ denote the matrix consisting of the first $n-1$ rows of $X$, so that 
\[
X = 
\begin{bmatrix}
\tX\\
X_n^T
\end{bmatrix},
\]
treating $X_n$ as a column vector. Let $\tilde{Y}_{n+1}$ and $\tilde{\sigma}$ be the predicted value of $Y_{n+1}$ and the estimated value of $\sigma$ if we use only the first $n-1$ data points. Then by the formula \eqref{predint} for the prediction interval, it is easy to see that 
\begin{align*}
&\lambda(P_n(X_{n+1}; L_n)\Delta P_{n-1}(X_{n+1}; L_{n-1})) \\
&\le |\hat{Y}_{n+1} - \tilde{Y}_{n+1}| + 2z_{1-\frac{\alpha}{2}}\biggl|\hat{\sigma}\sqrt{1+X_{n+1}^T(X^TX)^{-1}X_{n+1}} - \tilde{\sigma}\sqrt{1+X_{n+1}^T(\tX^T\tX)^{-1}X_{n+1}}\biggr|\\
&\le |\hat{Y}_{n+1} - \tilde{Y}_{n+1}| + 2z_{1-\frac{\alpha}{2}} |\hat{\sigma}- \tilde{\sigma}|\sqrt{1+X_{n+1}^T(X^TX)^{-1}X_{n+1}}\\
&\qquad + 2z_{1-\frac{\alpha}{2}}\tilde{\sigma}\biggl|\sqrt{1+X_{n+1}^T(X^TX)^{-1}X_{n+1}} - \sqrt{1+X_{n+1}^T(\tX^T\tX)^{-1}X_{n+1}}\biggr|.
\end{align*}
Our task, now, is to compute upper bounds on the expected values of the three terms above. Let us denote the three terms by $T_1$, $T_2$ and $T_3$. We will make several uses of the identity
\begin{align}\label{xyiden}
\sqrt{x} - \sqrt{y} = \frac{x-y}{\sqrt{x}+\sqrt{y}}.
\end{align}
First, note that 
\[
X^TX = \tX^T\tX + X_nX_n^T
\]
By the well known formula for the inverse of a rank-one perturbation, this gives
\begin{align}\label{xtxinv}
(X^TX)^{-1} &= (\tX^T \tX)^{-1} - \frac{(\tX^T\tX)^{-1}X_n X_n^T(\tX^T \tX)^{-1}}{1+ X_n^T(\tX^T\tX)^{-1}X_n}.
\end{align}
Thus, we get
\begin{align*}
|X_{n+1}^T(X^TX)^{-1}X_{n+1} - X_{n+1}^T(\tX^T\tX)^{-1}X_{n+1}| &\le (X_{n+1}^T(\tX^T\tX)^{-1} X_n)^2. 
\end{align*}
By the above inequality and the identity \eqref{xyiden},
\begin{align*}
&\biggl|\sqrt{1+X_{n+1}^T(X^TX)^{-1}X_{n+1}} - \sqrt{1+X_{n+1}^T(\tX^T\tX)^{-1}X_{n+1}}\biggr|\\
&\le |(1+X_{n+1}^T(X^TX)^{-1}X_{n+1}) - (1+X_{n+1}^T(\tX^T\tX)^{-1}X_{n+1})|\\
&\le \|(\tX^T\tX)^{-1}\|^2 \|X_{n+1}\|^2\|X_n\|^2.
\end{align*}
Thus, by an application of H\"older's inequality,
\begin{align*}
&\E\biggl(\tilde{\sigma}\biggl|\sqrt{1+X_{n+1}^T(X^TX)^{-1}X_{n+1}} - \sqrt{1+X_{n+1}^T(\tX^T\tX)^{-1}X_{n+1}}\biggr|\biggr)\\
&\le \E(\tilde{\sigma}\|(\tX^T\tX)^{-1}\|^2 \|X_{n+1}\|^2\|X_n\|^2)\\
&\le [\E(\tilde{\sigma}^4)\E(\|(\tX^T\tX)^{-1}\|^8)\E(\|X_{n+1}\|^8) \E(\|X_n\|^8)]^{\frac{1}{4}}.
\end{align*}
By the sub-Gaussian tail assumption, $\E(\|X_1\|^8)$ is finite, and by Lemma \ref{predictlmm3}, 
\[
\E(\|(\tX^T\tX)^{-1}\|^8)\le Cn^{-8}.
\]
By Lemma \ref{predictlmm5} with $n-1$ instead of $n$, $\E(\tilde{\sigma}^4)\le C$. Thus, the left side in the preceding display is bounded above by $Cn^{-2}$. This proves that 
\begin{align}\label{t3ineq}
\E(T_3)\le \frac{C}{n^2}.
\end{align}
Let $\tilde{Y}$ be the vector consisting of the first $n-1$ components of $Y$. By the formula \eqref{hatsigma},
\begin{align*}
(n-p-1)\tilde{\sigma}^2 = \tilde{Y}^T\tilde{Y}- (\tX^T \tY)^T(\tX^T\tX)^{-1}(\tX^T\tY),
\end{align*}
and therefore, 
\begin{align*}
(n-p)\hat{\sigma}^2 &= Y^TY - (X^TY)^T(X^TX)^{-1}(X^TY)\\
&= \tY^T \tY + Y_n^2 - (\tX^T \tY + X_n Y_n)^T(X^TX)^{-1}(\tX^T \tY + X_n Y_n)\\
&= \tY^T\tY + Y_n^2 - (\tX^T\tY)^T(X^TX)^{-1}(\tX^T\tY) - Y_n^2 X_n^T(X^TX)^{-1}X_n\\
&\qquad - 2Y_n X_n^T(X^TX)^{-1}\tX^T \tY\\
&= (n-p-1)\tilde{\sigma}^2 + Y_n^2 - (\tX^T\tY)^T((X^TX)^{-1} - (\tX^T \tx)^{-1})(\tX^T\tY)\\
&\qquad - Y_n^2 X_n^T(X^TX)^{-1}X_n - 2Y_n X_n^T(X^TX)^{-1}\tX^T \tY.
\end{align*}
This shows that
\begin{align*}
(n-p)|\hat{\sigma}^2 - \tilde{\sigma}^2| &\le \tilde{\sigma}^2 + Y_n^2 + \|(X^TX)^{-1} - (\tX^T\tX)^{-1}\|\|\tX\|^2 \|\tY\|^2\\
&\qquad + \|(X^TX)^{-1}\|\|X_n\|^2 Y_n^2 + 2\|(X^TX)^{-1}\|\|\tX\|\|\tY\|\|X_n\||Y_n|.
\end{align*}
But by the identity \eqref{xtxinv}, 
\begin{align}\label{xxineq2}
\|(X^TX)^{-1} - (\tX^T\tX)^{-1}\| &\le \|(\tX^T\tX)^{-1}\|^2 \|X_n\|^2.
\end{align}
Using this in the previous display, we get
\begin{align}\label{t2main}
(n-p)|\hat{\sigma}^2 - \tilde{\sigma}^2| &\le \tilde{\sigma}^2 + Y_n^2 + \|(\tX^T\tX)^{-1}\|^2 \|X_n\|^2\|\tX\|^2 \|\tY\|^2\notag\\
&\qquad + \|(X^TX)^{-1}\|\|X_n\|^2 Y_n^2 + 2\|(X^TX)^{-1}\|\|\tX\|\|\tY\|\|X_n\||Y_n|.
\end{align}
Now recall that for each $k\ge 1$, we have the following inequalities as consequences of the sub-Gaussian tail assumption, Lemma \ref{predictlmm1}, Lemma \ref{predictlmm3}, Lemma \ref{predictlmm5}, and inequality \eqref{ymoment}, provided that $n\ge C_1(k)$:
\begin{align}\label{allbounds}
&\E[|Y_n|^k]\le C(k), \ \ \E[\|X_n\|^k]\le C(k), \ \ \E[\|(X^TX)^{-1}\|^k] \le C(k) n^{-k}, \ \ \E(\tilde{\sigma}^{\pm k}) \le C(k),\notag\\
&\E[\|(\tX^T\tX)^{-1}\|^k] \le C(k) n^{-k}, \ \ \E[\|\tX\|^k]\le C(k) n^{\frac{k}{2}}, \ \ \E[\|\tY\|^k]\le C(k) n^{\frac{k}{2}}.
\end{align}
Using these inequalities and several applications of H\"older's inequality, the inequality \eqref{t2main} yields
\begin{align}\label{sigmadiff}
\E[|\hat{\sigma}^2 - \tilde{\sigma}^2|^k] &\le \frac{C(k)}{n^k}.
\end{align}
Now note that by equation \eqref{xyiden},
\begin{align*}
|\hat{\sigma}-\tilde{\sigma}| &= \frac{|\hat{\sigma}^2-\tilde{\sigma}^2|}{\hat{\sigma}+\tilde{\sigma}}\le \frac{|\hat{\sigma}^2-\tilde{\sigma}^2|}{\hat{\sigma}}.
\end{align*}
Thus, we get
\begin{align*}
\E(T_2) &\le\E\biggl[\frac{|\hat{\sigma}^2- \tilde{\sigma}^2|}{\hat{\sigma}}\sqrt{1+X_{n+1}^T(X^TX)^{-1}X_{n+1}}\biggr] \\
&\le \biggl[\E\biggl(\frac{|\hat{\sigma}^2- \tilde{\sigma}^2|^2}{\hat{\sigma}^2}\biggr) \E(1+X_{n+1}^T(X^TX)^{-1}X_{n+1})\biggr]^{\frac{1}{2}}.
\end{align*}
Using the bounds displayed in equation \eqref{allbounds}, we get
\begin{align*}
\E(1+X_{n+1}^T(X^TX)^{-1}X_{n+1}) &\le 1 + \E(\|(X^TX)^{-1}\|\|X_{n+1}\|^2)\\
&\le 1 + [\E(\|(X^TX)^{-1}\|^2) \E(\|X_{n+1}\|^4)]^{\frac{1}{2}}\le C,
\end{align*}
and combining with equation \eqref{sigmadiff}, 
\begin{align*}
\E\biggl(\frac{|\hat{\sigma}^2- \tilde{\sigma}^2|^2}{\hat{\sigma}^2}\biggr) &\le [\E(|\hat{\sigma}^2 - \tilde{\sigma}^2|^4)\E(\hat{\sigma}^{-4})]^{\frac{1}{2}}\\
&\le \frac{C}{n^2}.
\end{align*}
Thus, we get
\begin{align}\label{t2ineq}
\E(T_2)\le \frac{C}{n}.
\end{align}
Finally, note that 
\begin{align*}
|\hat{Y}_{n+1} - \tilde{Y}_{n+1}| &= |X_{n+1}^T(X^TX)^{-1}X^TY - X_{n+1}^T (\tX^T\tX)^{-1}\tX^T\tY|\\
&= |X_{n+1}^T(X^TX)^{-1}(\tX^T \tY + X_n Y_n) - X_{n+1}^T (\tX^T\tX)^{-1}\tX^T\tY|\\
&\le |X_{n+1}^T((X^TX)^{-1}-(\tX^T\tX)^{-1})\tX^T\tY| + |Y_n||X_{n+1}^T(X^TX)^{-1}X_n|\\
&\le \|(X^TX)^{-1}-(\tX^T\tX)^{-1}\|\|X_{n+1}\|\|\tX\|\|\tY\| \\
&\qquad + |Y_n|\|X_{n+1}\|\|X_n\|\|(X^TX)^{-1}\|.
\end{align*}
Applying inequality \eqref{xxineq2} to the first term on the right, we get
\begin{align*}
|\hat{Y}_{n+1} - \tilde{Y}_{n+1}| &\le \|(\tX^T\tX)^{-1}\|^2 \|X_n\|^2\|X_{n+1}\|\|\tX\|\|\tY\| \\
&\qquad + |Y_n|\|X_{n+1}\|\|X_n\|\|(X^TX)^{-1}\|.
\end{align*}
Now applying the bounds from equation \eqref{allbounds} and several applications of H\"older's inequality, we get
\begin{align}\label{t1ineq}
\E(T_1) = \E|\hat{Y}_{n+1} - \tilde{Y}_{n+1}|  &\le   \frac{C}{n}.
\end{align}
The proof is completed by combining the inequalities \eqref{t3ineq}, \eqref{t2ineq} and \eqref{t1ineq}.
\end{proof}
\end{ex}

\subsection{A subtle point}\label{subtle}
Theorem \ref{mainthm} says that if the error bound is small, then 
\begin{align}\label{mainapprox}
\mu(A(X_1,\ldots,X_n)) \approx \frac{1}{n}\sum_{i=1}^n 1_{\{X_i \in A'(X_1,\ldots,X_{i-1},X_{i+1},\ldots,X_n)\}}
\end{align}
with high probability.  That is, the right side can be used to estimate the left side, in case we have the data $X_1,\ldots,X_n$ and we know $A$,  but we do not know $\mu$. A subtle but important remark is that the smallness of the error bound does not imply, however, that the random variable $\mu(A(X_1,\ldots,X_n))$ is concentrated near a deterministic value. In other words, this is not a standard concentration of measure result. This anomaly can arise in high dimensional settings (where the sample size is comparable to dimension of the space $S$ in which the $X_i$'s take value), as demonstrated by the following example due to David Aldous.

Let $n$ be a large number. Let $\mu$ be the probability measure on $\R^n$ described as follows. With probability $\frac{1}{n}$, choose the origin; with probability $1-\frac{1}{n}$, choose a point uniformly from the unit sphere $\mathbb{S}^{n-1}$. Let $X_1,\ldots,X_n$ be $n$ i.i.d.~points from this distribution (so that sample size $n$ is the same as the dimension $n$). Let $A(X_1,\ldots,X_n)$ be the set of all $x\in \R^n$ that are within distance $\frac{1}{2}(1+\sqrt{2})$ from at least one point among $X_1,\ldots,X_n$. Define $A'$ and $A''$ analogously.

We claim that in this example, the error bound from Theorem \ref{mainthm} is small (and therefore, equation \eqref{mainapprox} holds with high probability), but $\mu(A(X_1,\ldots,X_n))$ is not concentrated near a deterministic value.

To see this, note that if $X$ and $Y$ are independently and uniformly chosen from $\mathbb{S}^{n-1}$, then $\|X\| = 1 + O(n^{-\frac{1}{3}})$ and $\|X-Y\|= \sqrt{2} + O(n^{-\frac{1}{3}})$ with probability $1-O(e^{-n^{\frac{1}{3}}})$. These follow from easy probabilistic arguments. Now let $N$ be the number of $X_i$'s that are equal to $0$. Then $N\sim Binomial(n, \frac{1}{n})$, and thus, $N$ is approximately a $Poisson(1)$ random variable.

Suppose that $N$ turns out to be zero. Then all points in the sample are uniformly drawn from the sphere. Thus, if $X_{n+1}$ is a new sample drawn from $\mu$, then with probability $1-\frac{1}{n}$, $X_{n+1}$ is uniformly drawn from $\mathbb{S}^{n-1}$, which implies that $\min_{1\le i\le n}\|X_i-X_{n+1}\| \approx \sqrt{2}$ with high probability. Thus, if $N=0$, then $\mu(A(X_1,\ldots,X_n)) \approx 0$. On the other hand, if $N\ge 1$, then at least one of the $X_i$'s is zero. Thus, in this case, $\min_{1\le i\le n}\|X_i-X_{n+1}\| \approx 1$ with high probability. This implies that if $N\ge 1$, then $\mu(A(X_1,\ldots,X_n))\approx 1$. 

To summarize, we have shown that $\mu(A(X_1,\ldots,X_n)) \approx 0$ with probability $\approx e^{-1}$, and $\mu(A(X_1,\ldots,X_n)) \approx 1$ with probability $\approx 1-e^{-1}$. In particular, $\mu(A(X_1,\ldots,X_n))$ is not concentrated near a deterministic value.

Next, let us argue that the error bound from Theorem \ref{mainthm} is small. The $\theta/n$ term is small anyway, since $\theta \le \frac{1}{4}$. Next, note that if $N=0$, then $N$ remains zero for the subsample $X_1,\ldots,X_{n-1}$ as well. Thus, if $N=0$, then
\begin{align*}
\mu(A(X_1,\ldots,X_n)\Delta A'(X_1,\ldots,X_{n-1})) &= \mu(A(X_1,\ldots,X_n)) -  \mu(A'(X_1,\ldots,X_{n-1})) \approx 0.
\end{align*}
On the other hand, even if $N\ge 1$, it is very unlikely that $X_{n-1}=0$. Thus, it is highly likely that $N$ does not change even if we remove $X_{n-1}$ from the sample, and the above identity continues to hold. Since the left side is bounded by $1$, this lets us conclude that $\delta' \approx 0$. By a similar argument, $\delta''\approx 0$.

\bibliographystyle{abbrvnat}

\bibliography{myrefs}

\end{document}